\documentclass[a4paper,11pt]{article}
\usepackage[latin1]{inputenc}
\usepackage[english]{babel}
\usepackage{amsmath}
\usepackage{amsfonts}
\usepackage{amssymb}
\usepackage{epsfig}
\usepackage{amsopn}
\usepackage{amsthm}
\usepackage{color}
\usepackage{graphicx}
\usepackage{subfigure}
\usepackage{enumerate}
\usepackage{tikz-cd}
\newtheorem{theorem}{Theorem}[section]

\newtheorem{lemma}[theorem]{Lemma}
\newtheorem{proposition}[theorem]{Proposition}
\newtheorem{definition}[theorem]{Definition}

\newtheorem*{theorem*}{Theorem}
\newtheorem*{lemma*}{Lemma}
\newtheorem*{remark*}{Remark}
\newtheorem*{definition*}{Definition}
\newtheorem*{proposition*}{Proposition}
\newtheorem*{corollary*}{Corollary}
\numberwithin{equation}{section}
\newcommand{\real}{\mathbb{R}}

\let\ced=\c         

\def\x{\xi}

\def\qed{\,\unskip\kern 6pt \penalty 500
\raise -2pt\hbox{\vrule \vbox to8pt{\hrule width 6pt
\vfill\hrule}\vrule}\par}
\definecolor{darkblue}{rgb}{0.05, .05, .65}
\definecolor{darkgreen}{rgb}{0.1, .65, .1}
\definecolor{darkred}{rgb}{0.8,0,0}
\newcommand{\beqn}{\begin{equation}}
\newcommand{\eeqn}{\end{equation}}
\newcommand{\bear}{\begin{eqnarray}}
\newcommand{\eear}{\end{eqnarray}}
\newcommand{\bean}{\begin{eqnarray*}}
\newcommand{\eean}{\end{eqnarray*}}
\begin{document}


\title{\huge \bf Eternal solutions in exponential self-similar form for a quasilinear reaction-diffusion equation with critical singular potential}

\author{
\Large Razvan Gabriel Iagar\,\footnote{Departamento de Matem\'{a}tica
Aplicada, Ciencia e Ingenieria de los Materiales y Tecnologia
Electr\'onica, Universidad Rey Juan Carlos, M\'{o}stoles,
28933, Madrid, Spain, \textit{e-mail:} razvan.iagar@urjc.es},\\
[4pt] \Large Marta Latorre\,\footnote{Departamento de Matem\'{a}tica
Aplicada, Ciencia e Ingenieria de los Materiales y Tecnologia
Electr\'onica, Universidad Rey Juan Carlos, M\'{o}stoles,
28933, Madrid, Spain, \textit{e-mail:} marta.latorre@urjc.es},
\\[4pt] \Large Ariel S\'{a}nchez,\footnote{Departamento de Matem\'{a}tica
Aplicada, Ciencia e Ingenieria de los Materiales y Tecnologia
Electr\'onica, Universidad Rey Juan Carlos, M\'{o}stoles,
28933, Madrid, Spain, \textit{e-mail:} ariel.sanchez@urjc.es}\\
[4pt] }
\date{}
\maketitle

\begin{abstract}
We prove existence and uniqueness of self-similar solutions with exponential form
$$
u(x,t)=e^{\alpha t}f(|x|e^{-\beta t}), \qquad \alpha, \ \beta>0
$$
to the following quasilinear reaction-diffusion equation
$$
\partial_tu=\Delta u^m+|x|^{\sigma}u^p,
$$
posed for $(x,t)\in\real^N\times(0,T)$, with $m>1$, $1<p<m$ and $\sigma=-2(p-1)/(m-1)$ and in dimension $N\geq2$, the same results holding true in dimension $N=1$ under the extra assumption $1<p<(m+1)/2$. Such self-similar solutions are usually known in literature as \emph{eternal solutions} since they exist for any $t\in(-\infty,\infty)$. As an application of the existence of these eternal solutions, we show existence of \emph{global in time weak solutions} with any initial condition $u_0\in L^{\infty}(\real^N)$, and in particular that these weak solutions remain compactly supported at any time $t>0$ if $u_0$ is compactly supported.
\end{abstract}

\


\smallskip


\section{Introduction}

This paper is devoted to the study of the following quasilinear reaction-diffusion equation with critical singular potential
\begin{equation}\label{eq1}
\partial_tu=\Delta u^m+|x|^{\sigma}u^p, \qquad (x,t)\in\real^N\times(0,\infty),
\end{equation}
where $N\geq2$ and the exponents $m$, $p$, $\sigma$ satisfy the assumptions
\begin{equation}\label{range.exp}
m>1, \qquad 1<p<m, \qquad \sigma=\sigma_*:=-\frac{2(p-1)}{m-1}.
\end{equation}
We also consider Eq. \eqref{eq1} posed in dimension $N=1$ in the same range of exponents \eqref{range.exp} together with the extra assumption
\begin{equation}\label{range.exp1}
1<p<\frac{m+1}{2}, \qquad {\rm that \ is} \qquad \sigma_*>-1.
\end{equation}
As we see, this equation involves a singular potential as a weight on the reaction term and this precise value $\sigma=\sigma_*$ is critical with respect to the dynamics of the equation, as we will find out in the current work. Our goals in the present paper are, on the one hand, to establish the existence and uniqueness (up to a rescaling) of a compactly supported self-similar solution in exponential form which is global in time (that is, does not blow up in finite time) and, on the other hand, to give a theory of existence of global in time weak solutions in suitable functional spaces, a fact that it is always challenging when dealing with singular potentials and will be achieved here by employing the exponential self-similar solutions as upper barriers to prevent finite time blow-up.

The main mathematical interest of equations of the type of Eq. \eqref{eq1} is given by the double competition taking place between its two terms in the right hand side, in order to govern the evolution:

$\bullet$ on the one hand, we observe a general competition between a diffusion term of porous medium type and a source term. The effects of this competition, in the autonomous case $\sigma=0$, are by now quite well understood, and in particular, in the range we are concerned with, $m>1$ and $1<p<m$, finite time blow-up holds true for any solution (see for example \cite[Chapter IV]{S4}), that is, there exists $T\in(0,\infty)$ such that $u(t)\in L^{\infty}(\real^N)$ for any $t\in(0,T)$ but $u(T)\not\in L^{\infty}(\real^N)$. More precise characterizations (at least in dimension $N=1$) concerning blow-up set, blow-up of interfaces and detailed behavior of the solutions as $t\to T$ with an asymptotic pattern given by a self-similar solution in backward form are also given in the above mentioned reference.

$\bullet$ on the other hand, the presence of a singular potential of the form $|x|^{\sigma}$ with $\sigma<0$ introduces a new competition, between regions that are close to $x=0$ (where the singular potential is likely to produce a very strong reaction) and regions lying sufficiently far from the origin (where the singular potential is formally small). It has been seen since long time that the effects of this second competition on the dynamics of the solutions can be very striking. A famous example, which became a starting point for the study of singular potentials, comes from the celebrated paper by Baras and Goldstein \cite{BG84}, where it is shown that, if we let $m=p=1$ and $\sigma=-2$ in Eq. \eqref{eq1}, existence or non-existence of solutions (in the form of instantaneous blow-up at every point) is related to the famous optimal constant in Hardy's inequality. A similar conclusion about non-existence has been later established by Goldstein and other collaborators for $m=p<1$, that is, the fast-diffusion case \cite{GK03, GGK05, Ko04} or by Cabr\'e and Martel \cite{CM99} who replaced the weight $|x|^{-2}$ by more general, abstract weights $a(x)$ with suitable properties. Two of the authors also suggested (without giving a proof) a similar non-existence range for Eq. \eqref{eq1} with $p=m>1$ and $\sigma=-2$ in \cite{IS20}.

The above is just an example of how interesting or unexpected the outcome of the second competition can be. It also motivated to consider singular potentials related to reaction-diffusion equations, and in particular the study of the so-called \emph{Hardy equation} in the semilinear case $m=1$, $p=1$ and $-2<\sigma<0$, strongly developed in recent years, as shown by the growing number of works devoted to the subject, see for example \cite{BSTW17, BS19, CIT21a, CIT21b, T20, HT21}. With respect to similarity solutions, again in the semi-linear case $m=1$ but with $p>1$, Filippas and Tertikas \cite{FT00} gave a classification of them for $\sigma>-2$ (that could be both positive and negative), analyzing also the behavior near the blow-up time. Their study has been completed by Mukai and Seki \cite{MS21} with a study of \emph{blow-up of Type II}, that is, a type of blow-up achieved for $p$ sufficiently large and with variable blow-up rates. The quasilinear case $m>1$, but always in the range $p>m$, and with $\sigma>-2$ (including both positive or negative values) has been addressed by Qi \cite{Qi98} and then Suzuki \cite{Su02}, two works devoted to establish the Fujita-type exponent and then to analyze for which initial conditions $u_0$ solutions can be global in time. Later on, Andreucci and Tedeev established blow-up rates in \cite{AT05}.

However, we have noticed some years ago that works dealing with the range $m>1$ and $p\in(1,m)$ were missing from literature either when $\sigma>0$ or when $\sigma<0$. This gave us the motivation to start investigating this case, and the results obtained within several recent works were quite unexpected. As a first outcome, it seems that (as we have shown at the level of self-similar solutions), the occurence of finite time blow-up or not depends strongly on the sign of the following constant
\begin{equation}\label{const.L}
L:=\sigma(m-1)+2(p-1).
\end{equation}
Indeed, the authors gave very recently in \cite{ILS22b} (following previous results in \cite{IS19, IS21} restricted to $\sigma>0$ and dimension $N=1$) a classification of the self-similar blow-up patterns \textbf{for} $\mathbf{L>0}$, that is, solutions of the form (known as \emph{backward self-similarity})
\begin{equation}\label{backward}
u(x,t)=(T-t)^{-\alpha}f(|x|(T-t)^{\beta}), \qquad \alpha=\frac{\sigma+2}{L}, \qquad \beta=\frac{m-p}{L},
\end{equation}
showing that the form of the self-similar profiles $f$, and thus of the blow-up sets and rates of the solutions, strongly depends on the magnitude of $\sigma$. In particular, for $\sigma\in(-2,0)$, it is shown that blow-up occurs always simultaneously (that is, $u(x,t)\to\infty$ as $t\to T$, for any $x\in\real^N$).

On the contrary, in another recent work \cite{IMS21}, {\bf the range} $\mathbf{L<0}$ has been considered (for $\sigma>-2$), which directly implies that $-2<\sigma<\sigma_*<0$, with $\sigma_*$ defined in \eqref{range.exp}. It is then proved therein that there exists a unique self-similar solution which is global in time, in the form (known as \emph{forward self-similarity})
\begin{equation}\label{forward}
u(x,t)=t^{\alpha}f(|x|t^{-\beta}), \qquad \alpha=-\frac{\sigma+2}{L}, \qquad \beta=-\frac{m-p}{L},
\end{equation}
which obviously lives for any $t>0$. We also explain in \cite{IMS21} how the existence of this solution prevents blow-up of general solutions, and in fact a theory of existence of weak solutions which are global in time can be done in that case along the lines of the last section of the current work. We furthermore recently extended the study of global solutions to the limit case $\sigma=-2$, which gives obviously $L<0$ for any $p\in(1,m)$, showing in \cite{IS22c} that in this limit case the solution, also unique if asking to be compactly supported, has an integrable singularity at $x=0$.

In view of these precedents and previous comments, and in order to complete the theory of the case $\sigma<0$, we are thus left to consider {\bf the case} $\mathbf{L=0}$, that is, $\sigma=\sigma_*$, which is the aim of the present paper. It is now the right moment to describe more precisely the results of this work.

\medskip

\noindent \textbf{Main results.} As we have explained before, we are looking for self-similar solutions to Eq. \eqref{eq1} in the range of exponents \eqref{range.exp}. It is straightforward to notice that there cannot exist self-similar solutions in any of the previous forms \eqref{forward} or \eqref{backward}, as when inserting these ansatz in Eq. \eqref{eq1}, the time-dependent part cannot vanish, its exponents in the three terms of Eq. \eqref{eq1} form an incompatible system in $\alpha$ and $\beta$ precisely due to the critical value $\sigma=\sigma_*$. We are thus left with considering a third form (known as \emph{exponential self-similarity}), namely
\begin{equation}\label{exp.SS}
u(x,t)=e^{\alpha t}f(|x|e^{-\beta t}),
\end{equation}
for $\alpha$ and $\beta$ to be found. Introducing the ansatz \eqref{exp.SS} into Eq. \eqref{eq1}, we readily find that in order for $u$ to be a solution, its self-similar exponents must satisfy $\alpha=2\beta/(m-1)$, while the profile $f$ solves the differential equation
\begin{equation}\label{SSODE}
(f^m)''(\xi)+\frac{N-1}{\xi}(f^m)'(\xi)-\alpha f(\xi)+\beta\xi f'(\xi)+\xi^{\sigma}f(\xi)^p=0,
\end{equation}
where $\xi=|x|e^{-\beta t}$. Moreover, the following (formal) calculation of the total mass of a solution at any time $t>0$
\begin{equation*}
M(t):=\int_{\real^N}u(x,t)\,dx=\int_{\real^N}e^{\alpha t}f(|x|e^{-\beta t})\,dx=e^{(\alpha+N\beta)t}\int_{\real}f(\xi)\xi^{N-1}\,d\xi,
\end{equation*}
together with the natural assumption and expectation that $M(t)$ increases with time, since we are dealing with an equation with a source term, give that $\alpha+N\beta>0$, which in particular implies that $\alpha>0$ and $\beta>0$, as $\alpha=2\beta/(m-1)$. We thus assume from now on that both $\alpha$ and $\beta$ are positive. As a further remark, let us notice that, if $f$ is a solution to \eqref{SSODE} with exponents $\alpha$, $\beta>0$, then the rescaled functions
\begin{equation}\label{resc}
f_{\lambda}(\xi)=\lambda f\left(\lambda^{-(m-1)/2}\xi\right), \qquad \lambda>0,
\end{equation}
are also solutions to \eqref{SSODE} and thus candidates for self-similar profiles. With these notation and conventions, we are ready to state our first result, completely classifying the self-similar profiles in exponential form.
\begin{theorem}\label{th.1}
There exists a \textbf{unique} exponent $\alpha^*\in(0,\infty)$ (and corresponding $\beta^*=(m-1)\alpha^*/2$) such that
\begin{enumerate}
\item For $\alpha=\alpha^*$, there exists a unique one-parameter family of compactly supported self-similar solutions in exponential form \eqref{exp.SS} such that their profile $f$ satisfies $f(0)>0$, with local behavior given by
\begin{equation}\label{beh.Q1}
f_K(\xi)\sim\left[K-\frac{(m-1)^2}{2m[N(m-1)-2(p-1)]}\xi^{2(m-p)/(m-1)}\right]^{1/(m-p)}, \qquad {\rm as} \ \xi\to0,
\end{equation}
and such that there exists $\xi_0\in(0,\infty)$ such that $f(\xi_0)=0$ and $(f^m)'(\xi_0)=0$. All these profiles are obtained from the profile $f_1$ by the rescaling \eqref{resc}.
\item For any $\alpha\in(\alpha^*,\infty)$, there exists a unique one-parameter family of self-similar solutions in exponential form \eqref{exp.SS} such that their profile $f$ satisfies $f(0)>0$, with local behavior as $\xi\to0$ given by \eqref{beh.Q1} and the following (unbounded) behavior at infinity
\begin{equation}\label{beh.P0}
f(\xi)\sim C(m,p,\alpha)\xi^{2/(m-1)}(\log\,\xi)^{-1/(p-1)}, \qquad {\rm as} \ \xi\to\infty,
\end{equation}
where $C(m,p,\alpha)>0$ is a constant that is made explicit in \eqref{const.P0}. Again, all these profiles can be obtained from a single one of them by the rescaling \eqref{resc}.
\item For any $\alpha\in(0,\alpha^*)$ there are no self-similar solutions in exponential form \eqref{exp.SS} with such exponent $\alpha$ (and corresponding $\beta$).
\end{enumerate}
\end{theorem}
Notice that we had to ``normalize" $f(\xi)$ by fixing $f(0)=1$ in order to have uniqueness not only of the exponent $\alpha^*$ but also of the profile in Theorem \ref{th.1}. Indeed, if dropping this fixed value $f(0)=1$, we obtain in reality a one-parameter family of self-similar profiles $f_{\lambda}$ according to \eqref{resc}, with the same properties with respect to compact support or behavior at infinity given by \eqref{beh.P0}, in each case. Let us also emphasize here on the behavior \eqref{beh.P0}, which involves a mixed scale of power and logarithms. This is a sometimes tricky feature of critical exponents, where logarithmic corrections to algebraic behaviors might occur, see also \cite{IL13b} for such a mixed behavior.

\medskip

\noindent \textbf{Remark. Eternal solutions.} Observe that solutions as in \eqref{exp.SS} can be defined also in backward sense, making them \emph{eternal}, that is, valid for any $t\in(-\infty,\infty)$. This is a very unusual case in the theory of parabolic equations, since in general the backward problem associated to a parabolic equation is ill-posed. Such \emph{eternal solutions} have been thus identified only in a few critical cases limiting between ranges with very different behavior. The term \emph{eternal solutions} seems to stem (up to our knowledge) from Daskalopoulos and Sesum \cite{DS06} where such solutions are obtained for the two-dimensional Ricci flow. Prior to that, exponential self-similarity has been met in the paper \cite{GPV00} dealing with the critical exponent $m_c=(N-2)/N$ of the fast diffusion equation in dimension $N\geq3$, which limits between conservation of mass and finite time extinction. Two of the authors recently classified eternal solutions to Eq. \eqref{eq1} in the ranges $m>1$, $0<p<1$ in \cite{IS22a}, respectively the fast diffusion range $0<m<m_c$ and $p>1$ in \cite{IS22b}, in both cases with the same critical value $\sigma=\sigma_*$ but which is there positive, in contrast to our case. We also mention here the eternal solutions obtained in \cite{IL13b} for an equation with $p$-Laplacian fast diffusion and gradient absorption. We believe that the availability of such eternal solutions is an interesting feature of critical exponents in parabolic PDEs. 

\noindent Let us also remark at this point that the rescaling \eqref{resc} acts on the eternal self-similar solutions as a translation in time. Indeed, if $\lambda=e^{\alpha t_0}$ for some $t_0\in\real$ and if we define 
$$
u_{\lambda}(x,t)=e^{\alpha t}f_{\lambda}(|x|e^{-\beta t}),
$$
where $f_{\lambda}$ is defined in \eqref{resc}, we readily observe that 
$$
u_{\lambda}(x,t)=e^{\alpha(t+t_0)}f(|x|e^{-\beta(t+t_0)}).
$$
We thus understand better the uniqueness of exponential self-similar solutions in Theorem \ref{th.1}, since in fact all the one-parameter family is composed by a translation in time of a single one.

\medskip

\noindent \textbf{Existence of general global in time solutions.} The self-similar solutions obtained in Theorem \ref{th.1} can be used in order to establish existence of general solutions to the Cauchy problem associated to Eq. \eqref{eq1} with suitable initial conditions $u_0(x)=u(x,0)$, $x\in\real^N$. In order to state these results, let us introduce the following functional space
\begin{equation}\label{Linfx}
L^{\infty}(\real^N;|x|^{2/(m-1)}):=\{g\in L^{\infty}_{\rm loc}(\real^N): |x|^{-2/(m-1)}g(x)\in L^{\infty}(\real^N\setminus B(0,1))\}
\end{equation}
of functions having a behavior limited by the exponent in \eqref{beh.P0} at space infinity. We also define below the notion of \emph{weak solution} that will be employed (which in some texts is called \emph{very weak}, as it involves integrating by parts twice in the diffusion term).
\begin{definition}\label{def.sol}
By a \emph{weak solution to Eq. \eqref{eq1}} we understand a function $u\in C((0,T):L^{1}_{\rm loc}(\real^N))$ for some $T>0$, which moreover satisfies the following assumptions:
\begin{itemize}
\item $$
u(t)\in L^{1}_{\rm loc}(\real^N), \qquad u^m(t)\in L^{1}_{\rm loc}(\real^N), \qquad |x|^{\sigma}u^p(t)\in L^1_{\rm loc}(\real^N), \qquad {\rm for \ any } \ t\in(0,T).
$$
\item $u$ is a solution in the sense of distributions to Eq. \eqref{eq1}, that means that for any $\varphi\in C_0^{2,1}(\real^N\times(0,T))$ and for any $t_1$, $t_2\in(0,T)$ with $t_1<t_2$ we have
\begin{equation}\label{weak}
\begin{split}
\int_{\real^N}u(t_2)\varphi(t_2)\,dx&-\int_{\real^N}u(t_1)\varphi(t_1)\,dx-\int_{t_1}^{t_2}\int_{\real^N}u(t)\varphi_t(t)\,dx\,dt\\
&-\int_{t_1}^{t_2}\int_{\real^N}u^m(t)\Delta\varphi(t)\,dx\,dt=\int_{t_1}^{t_2}\int_{\real^N}|x|^{\sigma}u^p(x,t)\varphi(x,t)\,dx\,dt.
\end{split}
\end{equation}
\end{itemize}
We say that a function $u\in C([0,T):L^{1}_{\rm loc}(\real^N))$ for some $T>0$ is a \emph{weak solution to the Cauchy problem} with initial condition $u_0(x)$ if $u$ is a weak solution to Eq. \eqref{eq1} and the initial condition is taken in $L^1_{\rm loc}$ sense, that is
$$
\lim\limits_{t\to0}\left[\int_{\real^N}u(x,t)\varphi(x)\,dx-\int_{\real^N}u_0(x)\varphi(x)\,dx\right]=0,
$$
for any compactly supported test function $\varphi\in C_0(\real^N)$.
\end{definition}
With these definitions and notations in hand, we can now state the existence theorem for Eq. \eqref{eq1}.
\begin{theorem}\label{th.2}
Given $u_0\in L^{\infty}(\real^N)$, there exists at least a weak solution $u$ with
$$
u(t)\in L^{\infty}(\real^N;|x|^{2/(m-1)}), \qquad {\rm for \ any} \ t>0,
$$
to the Cauchy problem associated to Eq. \eqref{eq1} with initial condition $u(x,0)=u_0(x)$ for any $x\in\real^N$. If moreover the initial condition $u_0$ is compactly supported, then there exists a weak solution $u$ to the same Cauchy problem such that $u(t)$ is compactly supported for any $t>0$.
\end{theorem}
Notice that the solutions given by Theorem \ref{th.2} are global in time, and their existence will be obtained through a monotone approximation process, using the self-similar solutions obtained in Theorem \ref{th.1} as barriers from above in order to ensure that the limit function exists. We can thus say that the existence of global self-similar solutions established in Theorem \ref{th.1} \emph{prevents finite time blow-up} of general solutions to Eq. \eqref{eq1}. This is an interesting property which is, as explained at the beginning of the Introduction, completely due to the presence of the singular potential and the competition it involves between regions.

\section{The phase plane. Finite critical points}\label{sec.finite}

The proof of Theorem \ref{th.1} is based on a careful study of a phase plane associated to a planar, two-dimensional dynamical system into which the non-autonomous differential equation \eqref{SSODE} can be mapped. To make it precise, let us introduce the following new variables
\begin{equation}\label{PSchange1}
X(\xi)=m\xi^{-2}f(\xi)^{m-1}, \qquad Y(\xi)=m\xi^{-1}f(\xi)^{m-2}f'(\xi),
\end{equation}
together with the new independent variable of the system defined as
\begin{equation}\label{PSvar1}
\eta(\xi)=\frac{1}{m}\int_0^{\xi}\frac{\zeta}{f^{m-1}(\zeta)}\,d\zeta.
\end{equation}
By making an ``abuse of notation" and keeping the same letters for $X(\eta):=X(\eta(\xi))$, $Y(\eta):=Y(\eta(\xi))$, it follows after straightforward calculations that $(X,Y)$ solve the autonomous dynamical system
\begin{equation}\label{PSsyst1}
\left\{\begin{array}{ll}\dot{X}=X\left[(m-1)Y-2X\right],\\[1mm]
\dot{Y}=-Y^2-\beta Y+\alpha X-NXY-m^{(1-p)/(m-1)}X^{(m+p-2)/(m-1)},\end{array}\right.
\end{equation}
where derivatives are taken with respect to the variable $\eta$ introduced in \eqref{PSvar1}. This is the system whose phase plane we will analyze in the sequel. Notice at first that $X\geq0$ and that the line $\{X=0\}$ is invariant for the system \eqref{PSsyst1}, while $Y$ might change sign. Another important fact is that, in our range of exponents \eqref{range.exp}, we have
$$
1<\frac{m+p-2}{m-1}<2,
$$
an estimate which will be very helpful in the sequel. We now compute the finite critical points of the system \eqref{PSsyst1} and get only two: $P_0=(0,0)$ and $P_1=(0,-\beta)$. We analyze the flow of the system in a neighborhood of them below.
\begin{lemma}\label{lem.P0}
The linear approximation of the system \eqref{PSsyst1} in the neighborhood of the critical point $P_0$ has a one-dimensional stable manifold and (many) one-dimensional center manifolds. Any center manifold contains orbits entering $P_0$, and thus the point $P_0$ behaves like a stable node for orbits entering it from the half-plane $\{X>0\}$. The orbits entering $P_0$ on the center manifolds contain profiles with local behavior given by \eqref{beh.P0} as $\xi\to\infty$, with the explicit constant
\begin{equation}\label{const.P0}
C(m,p,\alpha)=\left(\frac{\alpha(m-1)}{2(p-1)}\right)^{1/(p-1)}m^{-(m+p-2)/[(m-1)(p-1)]}.
\end{equation}
\end{lemma}
\begin{proof}
The linear approximation of the system \eqref{PSsyst1} near $P_0$ has the matrix
$$
M(P_0)=\left(
         \begin{array}{cc}
           0 & 0 \\
           \alpha & -\beta \\
         \end{array}
       \right),
$$
hence the one-dimensional stable manifold (which is unique, see \cite[Theorem 3.2.1]{GH}) and the one-dimensional center manifolds (which may not be unique) correspond to the eigenvalues $\lambda_2=-\beta<0$, respectively $\lambda_1=0$. In order to apply the Local Center Manifold Theorem \cite[Theorem 1, Section 2.12]{Pe}, we have to perform a further change of variable to put the system \eqref{PSsyst1} into the canonical form. We thus set $V=\beta Y-\alpha X$ and obtain the system
\begin{equation}\label{PSsyst1.bis}
\left\{\begin{array}{ll}\dot{X}=\frac{m-1}{\beta}XV,\\[1mm]
\dot{V}=-\frac{1}{\beta}V^2-\beta V-\frac{mN-N+2m+2}{m-1}XV-\frac{(mN-N+2)\alpha}{m-1}X^2-m^{(1-p)/(m-1)}\beta X^{(m+p-2)/(m-1)}.\end{array}\right.
\end{equation}
Noticing that the vector field of the system \eqref{PSsyst1.bis} is of class $C^r$ with $r=(m+p-2)/(m-1)\in(1,2)$, we can apply \cite[Theorem 1, Section 2.12]{Pe} and look for an approximation of the center manifold of the form
\begin{equation}\label{interm1}
V=h(X)=aX^{\theta}+o(X^{\theta}), \qquad \theta>1,
\end{equation}
as allowed for example by \cite[Theorem 3, Section 2.5]{Carr}. Inserting the approximation \eqref{interm1} into the equation of the center manifold, we readily get that $\theta=(m+p-2)/(m-1)\in(1,2)$ and $a=-m$ by identifying the exponents of lowest order (which is $(m+p-2)/(m-1)$ since it is smaller than two). We thus infer that the center manifold has the local approximation
\begin{equation}\label{center}
V=-mX^{(m+p-2)/(m-1)}+o\left(X^{(m+p-2)/(m-1)}\right)
\end{equation}
and the flow on the center manifold, according to the Reduction Principle \cite[Theorem 2, Section 2.4]{Carr}, is given by
\begin{equation*}
\dot{X}=-\frac{2m}{\alpha}X^{(m+p-2)/(m-1)+1}+o(X^{(m+p-2)/(m-1)+1}),
\end{equation*}
which shows that the center manifolds are stable with respect to the flow. In order to go back to profiles, we first undo the change of variable to get back to $X$ and $Y$ and we infer from \eqref{center} that
$$
\beta Y-\alpha X=-mX^{(m+p-2)/(m-1)}+o\left(X^{(m+p-2)/(m-1)}\right),
$$
which in terms of profiles reads (dropping the small order from the notation for simplicity)
$$
\beta m\xi^{-1}f^{m-2}(\xi)f'(\xi)-\alpha m\xi^{-2}f^{m-1}(\xi)\sim-m^{(2m+p-3)/(m-1)}\xi^{-2(m+p-2)/(m-1)}f^{m+p-2}(\xi)
$$
or equivalently, after multiplying by $\xi f^{2-m}(\xi)/(m\beta)$,
\begin{equation}\label{interm2}
f'(\xi)-\frac{\alpha}{\beta}\xi^{-1}f(\xi)+\frac{1}{\beta}m^{(m+p-2)/(m-1)}\xi^{1-2(m+p-2)/(m-1)}f^p(\xi)\to0,
\end{equation}
with convergence as $\eta=\eta(\xi)\to\infty$. First of all, the equation \eqref{interm2} is integrable and we obtain the local behavior
\begin{equation}\label{interm3}
f(\xi)\sim\left(K+C(m,p,\alpha)^{-(p-1)}\log\,\xi\right)^{-1/(p-1)}\xi^{2/(m-1)},
\end{equation}
where $C(m,p,\alpha)$ is the constant defined in \eqref{const.P0}. We have to use now the fact that $X(\xi)\to0$ on the center manifolds, thus
$$
X(\xi)\sim m\left(K+C(m,p,\alpha)^{-(p-1)}\log\,\xi\right)^{-(m-1)/(p-1)}\to0,
$$
which shows that indeed the approximation in \eqref{interm3} has to be taken as $\xi\to\infty$. Then, the integration constant $K$ in \eqref{interm3} becomes negligible with respect to the $\log\,\xi$ term, leading to the behavior \eqref{beh.P0}, as stated.
\end{proof}
The analysis near $P_1$ is immediate, since this point is hyperbolic. We borrow ideas from the analogous \cite[Lemma 2.2]{ILS22b}.
\begin{lemma}\label{lem.P1}
The critical point $P_1=(0,-\beta)$ is a saddle point. The unique orbit going out of $P_1$ lies on the invariant line $\{X=0\}$, while the unique orbit entering $P_1$ contains profiles with an interface behavior at some $\xi_0\in(0,\infty)$. The behavior at the interface is given by
\begin{equation}\label{beh.P1}
f(\xi)\sim\left[C-\frac{\beta(m-1)}{2m}\xi^2\right]_{+}^{1/(m-1)}, \qquad {\rm as} \ \xi\to\xi_0=\sqrt{\frac{2mC}{\beta(m-1)}}, \ \xi<\xi_0,
\end{equation}
where $C>0$ is a free constant.
\end{lemma}
\begin{proof}
The linear approximation of the system \eqref{PSsyst1} in a neighborhood of $P_1$ has the matrix
$$
M(P_1)=\left(
         \begin{array}{cc}
           -(m-1)\beta & 0 \\
           \alpha+N\beta & \beta \\
         \end{array}
       \right),
$$
thus it is obvious that $P_1$ is a saddle point. The unstable manifold corresponds to the eigenvalue $\lambda_2=\beta>0$, with eigenvector $e_2=(0,1)$, thus it is fully contained in the line $\{X=0\}$, since the latter is invariant. The stable manifold corresponds to the eigenvalue $\lambda_1=-(m-1)\beta<0$. In terms of profiles, the local behavior near $P_1$ is deduced starting from the fact that $Y\to-\beta$, together with the fact that $X\to0$. If this limits would be taken as $\xi\to\infty$, the fact that
$$
\xi X'(\xi)=-2X(\xi)+(m-1)Y(\xi)
$$
together with an application of \cite[Lemma 2.9]{IL13a} for the function $X(\xi)$ would imply that there exists a sequence $\xi_k\to\infty$ such that $\xi_k X'(\xi_k)\to0$ as $k\to\infty$. Since $X(\xi_k)\to0$, we infer that $(m-1)Y(\xi_k)\to0$ as $k\to\infty$, which is a contradiction with the fact that $Y\to-\beta$. Thus the previous limits are taken, in terms of profiles, as $\xi\to\xi_0\in(0,\infty)$ from the left, which gives first that $f(\xi_0)=0$ and then
$$
(f^{m-1})'(\xi)\sim-\frac{\beta(m-1)\xi}{m}, \qquad {\rm as} \ \xi\to\xi_0,
$$
and the local behavior given by \eqref{beh.P1} follows by direct integration $(\xi,\xi_0)$.
\end{proof}

\section{Critical points at infinity}\label{sec.infty}

This rather technical section is needed in order to complete the local analysis of the phase plane associated to the system \eqref{PSsyst1}. In order to analyze these points, we pass to the Poincar\'e sphere by setting
$$
\overline{X}=\frac{X}{W}, \qquad \overline{Y}=\frac{Y}{W},
$$
and according to \cite[Theorem 1, Section 3.10]{Pe}, the critical points at infinity lie on the equator of the Poincar\'e sphere in variables $(\overline{X},\overline{Y},W)$, that is, $W=0$ and $\overline{X}^2+\overline{Y}^2=1$. Moreover, they are also solutions of the equation
\begin{equation*}
\begin{split}
&\overline{X}Q^*(\overline{X},\overline{Y},W)=\overline{Y}P^*(\overline{X},\overline{Y},W), \qquad {\rm where} \\ &P^*(\overline{X},\overline{Y},W):=W^2P\left(\frac{\overline{X}}{W},\frac{\overline{Y}}{W}\right), \qquad Q^*(\overline{X},\overline{Y},W):=W^2Q\left(\frac{\overline{X}}{W},\frac{\overline{Y}}{W}\right),
\end{split}
\end{equation*}
where $P$ and $Q$ are the components of the vector field of the system \eqref{PSsyst1}, that is, the right hand side of the two equations. In the case of the system \eqref{PSsyst1}, straightforward calculations give that the critical points satisfy the equations
$$
\overline{X}\overline{Y}(m\overline{Y}+(N-2)\overline{X})=0, \qquad \overline{X}^2+\overline{Y}^2=1,
$$
hence we find four critical points at infinity, namely
$$
Q_1=(1,0,0), \qquad Q_{2,3}=(0,\pm1,0), \qquad Q_4=\left(\frac{m}{\sqrt{m^2+(N-2)^2}},-\frac{N-2}{\sqrt{m^2+(N-2)^2}},0\right).
$$
We classify next the orbits connecting to these critical points, first in dimension $N\geq3$.

\subsection{Local analysis at infinity for $N\geq3$}\label{subsec.3}

Let us fix throughout this section $N\geq3$. According to \cite[Theorem 2, Section 3.10]{Pe}, the flow of the system \eqref{PSsyst1} in a neighborhood of the two critical points $Q_1$ and $Q_4$ that have nonzero $\overline{X}$ component is analyzed by projecting on the $X$-variable, that is, performing the change of variable
\begin{equation}\label{change.inf}
y=\frac{Y}{X}, \qquad z=\frac{1}{X}.
\end{equation}
Moreover, we notice that in the general framework of \cite[Theorem 2, Section 3.10]{Pe} we have to choose the minus sign, since we are dealing with $X$ as dominating variable and the first equation in \eqref{PSsyst1} gives in a neighborhood of $Q_1$ and $Q_4$ that
$$
\dot{X}=X^2\left[-2+(m-1)\frac{Y}{X}\right]<0,
$$
since either $Y/X\to0$ (at $Q_1$) or $Y/X\to-(N-2)/m<0$ (at $Q_4$). With this choice of sign and the change of variable \eqref{change.inf}, the system given in \cite[Theorem 2, Section 3.10]{Pe} writes in our case
\begin{equation}\label{PSsystinf.bis}
\left\{\begin{array}{ll}\dot{y}=-(N-2)y-my^2-\beta yz+\alpha z-m^{(1-p)/(m-1)}z^{(m-p)/(m-1)},\\[1mm]
\dot{z}=2z-(m-1)yz,\end{array}\right.
\end{equation}
and the points $Q_1$, respectively $Q_4$, are mapped into the critical points $(0,0)$, respectively $(-(N-2)/m,0)$ in the system \eqref{PSsystinf.bis}. Noticing that $(m-p)/(m-1)\in(0,1)$, we need to perform a further change of variable in order to convert the system \eqref{PSsystinf.bis} into one having a $C^1$ vector field. We thus let $w=z^{(m-p)/(m-1)}$ and finally get the system that we will use next
\begin{equation}\label{PSsystinf1}
\left\{\begin{array}{ll}\dot{y}=-(N-2)y-my^2-\beta yw^{(m-1)/(m-p)}+\alpha w^{(m-1)/(m-p)}-m^{(1-p)/(m-1)}w,\\[1mm]
\dot{w}=\frac{m-p}{m-1}\left(2w-(m-1)yw\right),\end{array}\right.
\end{equation}
By ``abuse of language", we will still refer the two critical points $(0,0)$ and $(-(N-2)/m,0)$ of the latter system as $Q_1$ and $Q_4$.
\begin{lemma}\label{lem.Q1}
The critical point $Q_1=(0,0)$ of the system \eqref{PSsystinf1} is a saddle point. The unique orbit contained in its unstable manifold contains profiles with the local behavior \eqref{beh.Q1} as $\xi\to0$.
\end{lemma}
\begin{proof}
The linear approximation of the system \eqref{PSsystinf1} in a neighborhood of $Q_1$ has the matrix
$$
M(Q_1)=\left(
         \begin{array}{cc}
           -(N-2) & -m^{(1-p)/(m-1)} \\[1mm]
           0 & \frac{2(m-p)}{m-1} \\
         \end{array}
       \right),
$$
hence $Q_1$ is a saddle point, with one orbit in the stable manifold contained in the invariant line $\{w=0\}$, since the eigenvector corresponding to the eigenvalue $\lambda_1=-(N-2)<0$ is $e_1=(1,0)$, and one orbit in the unstable manifold corresponding to the eigenvalue $\lambda_2=2(m-p)/(m-1)$ and going out of $Q_1$ tangent to the eigenvector $e_2=((m-1)m^{(1-p)/(m-1)},-[N(m-1)-2(p-1)])$. It then follows that in the first approximation, the orbit going out of $Q_1$ satisfies
$$
\frac{y}{w}\sim-\frac{(m-1)m^{(1-p)/(m-1)}}{N(m-1)-2(p-1)}.
$$
Recalling that $w=z^{(m-p)/(m-1)}$ and undoing the change of variable \eqref{change.inf}, we obtain
\begin{equation}\label{interm4}
Y\sim-\frac{(m-1)m^{(1-p)/(m-1)}}{N(m-1)-2(p-1)}X^{(p-1)/(m-1)}.
\end{equation}
Notice that this is perfectly coherent with the condition $Y/X\to0$, since $(p-1)/(m-1)<1$. Putting \eqref{interm4} in terms of profiles gives
$$
m\xi^{-1}f^{m-2}(\xi)f'(\xi)\sim-\frac{(m-1)m^{(1-p)/(m-1)}}{N(m-1)-2(p-1)}\left(m\xi^{-2}f^{m-1}(\xi)\right)^{(p-1)/(m-1)},
$$
or equivalently, after obvious simplifications,
\begin{equation}\label{interm5}
f^{m-p-1}(\xi)f'(\xi)\sim-\frac{m-1}{m[N(m-1)-2(p-1)]}\xi^{(m-1-2(p-1))/(m-1)}.
\end{equation}
Assume now for contradiction that the equivalences \eqref{interm4}, respectively \eqref{interm5}, are taken as $\xi\to\xi_0\in(0,\infty)$ from the right. Then, on the one hand, the condition $X\to\infty$ as $\xi\to\xi_0$ converts into $f(\xi)\to+\infty$ as $\xi\to\xi_0$, $\xi>\xi_0$, thus we would get a vertical asymptote on the right. On the other hand, since now $\xi\to\xi_0$, we readily get from \eqref{interm5} that $(f^{m-p})'(\xi)\to C(\xi_0)\in(-\infty,0)$ as $\xi\to\xi_0$, which is a contradiction with the vertical asymptote. We thus deduce that \eqref{interm4} and \eqref{interm5} are taken as $\xi\to0$, hence we arrive to the local behavior \eqref{beh.Q1} by integration in \eqref{interm5}.
\end{proof}
The same system \eqref{PSsystinf1} allows for the local analysis in a neighborhood of $Q_4$, as follows.
\begin{lemma}\label{lem.Q4}
The critical point $Q_4$ is an unstable node for $N\geq3$. The orbits contained in its two-dimensional unstable manifold contain profiles with a vertical asymptote at $\xi=0$, and with the precise local behavior
\begin{equation}\label{beh.Q4}
f(\xi)\sim C\xi^{-(N-2)/m}, \qquad {\rm as} \ \xi\to0, \qquad C>0 \ {\rm free \ constant}.
\end{equation}
\end{lemma}
\begin{proof}
Since we identify $Q_4=(-(N-2)/m,0)$ in the system \eqref{PSsystinf1}, the linear approximation of the system in a neighborhood of $Q_4$ has the matrix
$$
M(Q_4)=\left(
         \begin{array}{cc}
           N-2 & -m^{(1-p)/(m-1)} \\[1mm]
           0 & \frac{(m-p)(mN-N+2)}{m(m-1)} \\
         \end{array}
       \right),
$$
with two positive eigenvalues. In terms of profiles, we notice that the orbits going out of $Q_4$ have $y\to-(N-2)/m$, hence, undoing the change of variable \eqref{change.inf}, we get
\begin{equation}\label{interm6}
\frac{Y}{X}=\frac{\xi f'(\xi)}{f(\xi)}\to-\frac{N-2}{m}.
\end{equation}
Assume for contradiction that the limit in \eqref{interm6} is taken as $\xi\to\xi_0\in(0,\infty)$, $\xi>\xi_0$. Then we obtain that
$$
\lim\limits_{\xi\to\xi_0}\frac{f'(\xi)}{f(\xi)}=-\frac{N-2}{m\xi_0},
$$
which by integration on $(\xi_0,\xi)$ for some $\xi>\xi_0$ sufficiently close to $\xi_0$ gives that
$$
f(\xi)\sim f(\xi_0)\exp\left[-\frac{N-2}{m\xi_0}(\xi-\xi_0)\right], \qquad {\rm as} \ \xi\to\xi_0.
$$
But this is a contradiction with the assumed fact that $X\to\infty$ as $\xi\to\xi_0$, since the latter implies that $f(\xi)\to\infty$ as $\xi\to\xi_0$. We are thus left with the behavior \eqref{interm6} as $\xi\to0$, which conducts to \eqref{beh.Q4} by direct integration.
\end{proof}
We are left with the local analysis of the flow in a neighborhood of the remaining points $Q_2$ and $Q_3$. In order to analyze them, we have to project on the $Y$ variable which is in their neighborhood the dominant one. We thus set
\begin{equation}\label{change.inf2}
x=\frac{X}{Y}, \qquad z=\frac{1}{Y}
\end{equation}
and apply the second part of \cite[Theorem 2, Section 3.10]{Pe} in order to obtain, after some direct calculations that we omit here, the system
\begin{equation}\label{interm7}
\left\{\begin{array}{ll}\pm\dot{x}=mx+\beta xz+\alpha x^2z+(N-2)x^2+m^{(1-p)/(m-1)}x^{(2m+p-3)/(m-1)}z^{(m-p)/(m-1)},\\[1mm]
\pm\dot{z}=z+\beta z^2+\alpha xz^2+Nxz+m^{(1-p)/(m-1)}x^{(m+p-2)/(m-1)}z^{(2m-p-1)/(m-1)},\end{array}\right.
\end{equation}
where the signs plus or minus have to be chosen according to the direction of the flow.
\begin{lemma}\label{lem.Q23}
The critical point $Q_2$ is an unstable node and the critical point $Q_3$ is a stable node. The orbits going out of $Q_2$ on its unstable manifold contain profiles with a change of sign at some point $\xi_0\in(0,\infty)$, in the sense that $f(\xi_0)=0$ and $(f^m)'(\xi_0)>0$, with the local behavior
$$
f(\xi)\sim C(\xi-\xi_0)^{1/m}, \qquad {\rm as} \ \xi\to\xi_0, \ \xi>\xi_0.
$$
The orbits entering $Q_3$ on its stable manifold contain profiles with a change of sign at some point $\xi_0\in(0,\infty)$ such that $f(\xi_0)=0$ and $(f^m)'(\xi_0)>0$, with the local behavior
$$
f(\xi)\sim C(\xi_0-\xi)^{1/m}, \qquad {\rm as} \ \xi\to\xi_0, \ \xi<\xi_0.
$$
\end{lemma}
Notice that we have to analyze in both cases the origin of the system \eqref{interm7}, but when working in a neighborhood of $Q_2$ we have to choose the minus sign, since $Y\to+\infty$ and thus decreases along the orbits, while when working in a neighborhood of $Q_3$ we have to choose the plus sign, since $Y\to-\infty$ and also decreases along the orbit but the direction of the flow as ``observed" from the critical point is reversed. We omit here the proof, as it is very similar to the one of \cite[Lemma 2.4]{IS22a} to which we refer the reader (an analogous proof is also given in detail in \cite[Lemma 2.6]{IS21}).

\subsection{Local analysis at infinity in dimension $N=2$ and $N=1$}\label{subsec.12}

In dimensions $N=2$ and $N=1$ there appear some noticeable differences in the analysis of the flow of the system \eqref{PSsystinf1} in a neighborhood of the critical points $Q_1$ and $Q_4$, as it can be readily seen from the signs of the eigenvalues of the matrices $M(Q_1)$ and $M(Q_4)$.

\medskip

\noindent \textbf{Dimension} $\mathbf{N=2}$. In this case the points $Q_1$ and $Q_4$ coincide. We keep labelling this point as $Q_1$ for simplicity, and the local behavior of the orbits going out of it is given below.
\begin{lemma}\label{lem.Q1Q4}
The critical point $Q_1=(0,0)$ of the system \eqref{PSsystinf1} in dimension $N=2$ is a saddle-node. There is a unique orbit going out of $Q_1$ into the phase plane which contains profiles with the local behavior \eqref{beh.Q1}. All the other orbits go out of $Q_1$ tangent to the $y$ axis into the region $\{Y<0\}$ and contain profiles such that
\begin{equation}\label{beh.Q12}
f(\xi)\sim D\left(-\ln\,\xi\right)^{1/m}, \qquad {\rm as} \ \xi\to0, \qquad D>0 \ {\rm free \ constant}.
\end{equation}
\end{lemma}
\begin{proof}[Sketch of the proof]
This result is analogous to the one of \cite[Lemma 4.1]{IS22a}, thus we will just give a sketch here for the sake of completeness. Let us notice that for $N=2$, we have
$$
M(Q_1)=M(Q_4)=\left(
         \begin{array}{cc}
           0 & -m^{(1-p)/(m-1)} \\[1mm]
           0 & \frac{2(m-p)}{m-1} \\
         \end{array}
       \right),
$$
with eigenvalues $\lambda_1=0$, $\lambda_2=2(m-p)/(m-1)>0$ and corresponding eigenvectors $e_1=(1,0)$ and $e_2=((m-1)m^{(1-p)/(m-1)},-2(m-p))$. The unique orbit going out of $Q_1$ on the stable manifold associated to the second eigenvalue contains profiles having the local behavior as in \eqref{beh.Q1}, as the proof of Lemma \ref{lem.Q1} can be reproduced for them. We then have center manifolds (that may not be unique) tangent to the $y$ axis according to the theory in \cite[Section 3.4]{GH}. In fact, according to the approximation theorem \cite[Theorem 3, Section 2.5]{Carr} which allows us to write the center manifold as
$$
w(y)=ay^2+o(y^2), \qquad a\in\real
$$
we get after straightforward calculations that $w(y)=o(y^2)$ and thus, on the center manifold, we get by neglecting lower order terms
$$
\frac{dy}{dw}\sim-\frac{m(m-1)}{2(m-p)}\frac{y^2}{w},
$$
hence
$$
y\sim\frac{2(m-p)}{m(m-1)}\frac{1}{\ln\,w}=\frac{2}{m\ln\,z}.
$$
We thus arrived to exactly the same situation as in the proof of \cite[Lemma 4.1]{IS22a}, since in our notation, variables $(y,z)$ correspond to variables $(Y,X)$ therein. The rest of the proof is based on the fact that the first two terms in the differential equation \eqref{SSODE} dominate over the last three ones over the trajectories contained in the center manifolds of $Q_1$ and thus we get the local behavior \eqref{beh.Q12} by equating in a first approximation
\begin{equation*}
(f^m)''(\xi)+\frac{N-1}{\xi}(f^m)'(\xi)\to0, \qquad {\rm as} \ \xi\to0,
\end{equation*}
and the details follow very closely the ones in the second part of the proof of \cite[Lemma 4.1]{IS22a}, to which we refer the reader.
\end{proof}

\medskip

\noindent \textbf{Dimension} $\mathbf{N=1}$. In this case the sign of $2-N$ changes and the critical point $Q_4=(1/m,0)$ lies now in the positive half-plane with respect to the $y$-axis. We have the following
\begin{lemma}\label{lem.Q1N1}
Let $N=1$ and the assumption \eqref{range.exp1} be in force. The critical point $Q_1=(0,0)$ of the system \eqref{PSsystinf1} is an unstable node and the critical point $Q_4=(1/m,0)$ is a saddle point. There is a unique orbit going out of $Q_1$ which contains profiles such that $f'(0)=0$, and more precisely with the local behavior \eqref{beh.Q1}. All the other orbits going out of $Q_1$ contain profiles such that $f(0)>0$ with any possible slope $f'(0)=B\neq0$. The only orbit going out of the critical point $Q_4$ contains profiles such that
\begin{equation}\label{beh.Q41}
f(\xi)\sim D\xi^{1/m}, \qquad {\rm as} \ \xi\to0, \ D>0 \ {\rm free \ constant}.
\end{equation}
\end{lemma}
\begin{proof}
The linear approximation of the system \eqref{PSsystinf1} in a neighborhood of $Q_1$ in dimension $N=1$ has the matrix
$$
M(Q_1)=\left(
         \begin{array}{cc}
           1 & -m^{(1-p)/(m-1)} \\[1mm]
           0 & \frac{2(m-p)}{m-1} \\
         \end{array}
       \right),
$$
with two positive eigenvalues $\lambda_1=1$ and $\lambda_2=2(m-p)/(m-1)$. We infer from the assumption \eqref{range.exp1} that
$$
2(m-p)>2\left(m-\frac{m+1}{2}\right)=m-1,
$$
whence $\lambda_2>\lambda_1$. Standard theory of the unstable manifold (see for example \cite[Theorem 19.11]{Amann}) implies the uniqueness of the orbit tangent to the eigenvector corresponding to the eigenvalue $\lambda_2$. In order to establish the local behavior of the orbits (and then, of the profiles contained in it), it is sufficient to integrate the linear approximation of the system \eqref{PSsystinf1}, which gives
\begin{equation*}
\frac{dy}{dw}\sim\frac{(m-1)[y-m^{(1-p)/(m-1)}w]}{2(m-p)w},
\end{equation*}
and by integration
\begin{equation}\label{orb1}
y\sim Cw^{(m-1)/2(m-p)}-\frac{(m-1)m^{(1-p)/(m-1)}}{m+1-2p}w,
\end{equation}
in a neighborhood of $(y,w)=(0,0)$. We next infer from \eqref{range.exp1} that $(m-1)/2(m-p)<1$, then the first term in \eqref{orb1} is the dominating one as $w\to0$. We thus have a unique orbit with local linear behavior, that is, the one corresponding to the integration constant $C=0$ in \eqref{orb1}, and we reach again the local behavior \eqref{interm4} in variables $(X,Y)$, which implies that the profiles contained in this orbit have the local behavior given in \eqref{beh.Q1}. All the other orbits going out of $Q_1$ have the local behavior given by
$$
y\sim Cw^{(m-1)/2(m-p)}, \qquad C\neq0,
$$
which is equivalent in the initial phase plane variables $(X,Y)$ to $Y=CX^{1/2}$, hence, in terms of profiles,
$$
\frac{2}{m-1}\left(f^{(m-1)/2}\right)'(\xi)\sim C, \qquad {\rm as} \ \xi\to0,
$$
which by integration gives functions with $f(0)>0$ and $f'(0)\neq0$, depending on the integration constant $C$, as claimed.

The linear approximation of the system \eqref{PSsystinf1} in a neighborhood of $Q_4=(1/m,0)$ in dimension $N=1$ has the matrix
$$
M(Q_4)=\left(
         \begin{array}{cc}
           -1 & -m^{(1-p)/(m-1)} \\[1mm]
           0 & \frac{(m-p)(m+1)}{m(m-1)} \\
         \end{array}
       \right),
$$
hence $Q_4$ is a saddle point. The orbit entering it lies on the invariant axis $\{w=0\}$ of the system \eqref{PSsystinf1}, while the unique orbit contained in its unstable manifold has $y\to 1/m$, which readily leads to the local behavior \eqref{beh.Q4} by integration after dropping the possibility that such a behavior is attained in a limit $\xi\to\xi_0\in(0,\infty)$ in the same way as in the end of the proof of Lemma \ref{lem.Q4}.
\end{proof}

\medskip

\noindent \textbf{Remark}. The analysis in this section shows that we are dealing with a \emph{transcritical bifurcation} (according to \cite{S73}, see also \cite[Section 3.4]{GH}), that is, an interchange of stability between the points $Q_1$ and $Q_4$ that takes place at $N=2$ in the system \eqref{PSsystinf1}.

\section{Proof of Theorem \ref{th.1}: existence}\label{sec.exist}

Throughout all this section, we assume that either $N\geq2$ and \eqref{range.exp} holds true, or $N=1$ and both \eqref{range.exp} and \eqref{range.exp1} hold true. We monitor the unique orbit going out of $Q_1$ containing profiles with the desired local behavior \eqref{beh.Q1}, whose uniqueness is ensured by Lemma \ref{lem.Q1} in dimension $N\geq3$, respectively Lemma \ref{lem.Q1Q4} in dimension $N=2$ and Lemma \ref{lem.Q1N1} in dimension $N=1$. Let us denote this orbit by $l(\alpha)$ for simplicity, as it depends on $\alpha>0$. As we immediately get from the classification of the critical points in Sections \ref{sec.finite} and \ref{sec.infty}, it has only three critical points to which it might connect, $P_0$, $P_1$ and $Q_3$, plus possible $\omega$-limits given by the Poincar\'e-Bendixon theory. We next classify this connection for $\alpha$ close to zero and for $\alpha$ very large. In the former, we have
\begin{proposition}\label{prop.low}
There exists $\alpha_1>0$ (and corresponding $\beta_1$) such that for any exponent $\alpha\in(0,\alpha_1)$, the orbit $l(\alpha)$ connects to the stable node $Q_3$.
\end{proposition}
\begin{proof}
Let first $N\geq3$. We begin with the system obtained by passing to the limit as $\alpha\to0$ in \eqref{PSsystinf1}. Since $\beta=(m-1)\alpha/2$, we also have $\beta\to0$ and in the limit, at least formally, we are left with the system
\begin{equation}\label{syst.lim0}
\left\{\begin{array}{ll}\dot{y}=-(N-2)y-my^2-m^{(1-p)/(m-1)}w,\\[1mm] \dot{w}=\frac{2(m-p)}{m-1}w-(m-p)yw.\end{array}\right.
\end{equation}
We start with the isocline given by $\dot{y}=0$, which is the parabola
\begin{equation}\label{interm8}
w=m^{(m+p-2)/(m-1)}\left[-y^2-\frac{N-2}{m}y\right],
\end{equation}
which connects the two critical points $Q_1$ and $Q_4$. The normal vector to this curve has the direction 
$$
\overline{n}=\left(m^{(m+p-2)/(m-1)}\left[-2y-\frac{N-2}{m}\right],-1\right),
$$ 
and thus the direction of the flow of the system \eqref{syst.lim0} on this curve is given by the sign of the scalar product of this normal vector with the vector field of the system, which gives the expression
$$
F(y,w)=-\frac{2(m-p)}{m-1}w+(m-p)yw<0,
$$
since we are only interested in the region where $w\geq0$ and $y<0$. It thus follows that the flow of the system goes in the increasing direction of the $w$ variable. On the other hand, it is obvious that the flow of the system \eqref{syst.lim0} on the line $\{y=0\}$ points towards the half-space $\{y<0\}$. Set
\begin{equation}\label{interm9}
\mathcal{D}:=\{(y,w):y<0,m^{-(m+p-2)/(m-1)}w>\max\{-y^2-(N-2)y/m,0\}\},
\end{equation}
which is then positively invariant for the flow of the system \eqref{syst.lim0}. Since the unique orbit going out of $Q_1$ and containing profiles with local behavior given by \eqref{beh.Q1} starts tangent to the eigenvector
$$
e_2=(-(m-1)m^{(1-p)/(m-1)},N(m-1)-2(p-1)),
$$
in a neighborhood of $Q_1=(0,0)$ it has a slope given by
$$
\frac{w}{y}\sim-\frac{N(m-1)-2(p-1)}{(m-1)m^{(1-p)/(m-1)}}<-\frac{N-2}{m^{(1-p)/(m-1)}}
$$
the latter being the slope at $Q_1$ of the isocline \eqref{interm8}. The last inequality follows easily from the fact that $N(m-1)-2(p-1)>(N-2)(m-1)$. This implies that the orbit going out of $Q_1$ enters the region $\mathcal{D}$ of the phase plane associated to the system \eqref{syst.lim0} and remains there. The same sentence remains true in dimensions $N=1$ and $N=2$ in a trivial way, since the isocline \eqref{interm8} in these dimensions does not enter the region where $w\geq0$ and $y<0$. We can visualize the outcome of the previous arguments in Figure 1 below.

\begin{figure}[h]\label{Fig1}
\centering
\begin{tikzpicture}[scale=1.3]
\draw [->, thick](-3,0) -- (1.5,0);
\draw [->, thick](0,-1.3) -- (0,2.8);
\node [right] at (1.5,0) {{ $y$}};
\node [above] at (0,2.8) {{ $w$}};

\draw [red, thick,->, >=stealth] (0,0)--(-0.3,1);

\draw [gray, thick,->, >=stealth] (-0.6,0.84)--(-0.6,1.34);
\draw [gray, thick,->, >=stealth] (-1.1,0.99)--(-1.1,1.49);
\draw [gray, thick,->, >=stealth] (-1.6,0.64)--(-1.6,1.14);

\draw [gray, thick,->, >=stealth] (0,1.4)--(-0.5,1.4);
\draw [gray, thick,->, >=stealth] (0,1.8)--(-0.5,1.8);
\draw [gray, thick,->, >=stealth] (0,2.2)--(-0.5,2.2);

\draw [blue, domain=-2.31:0, thick, samples=50] plot (\x, {-\x*\x-2*\x});

\node [below] at (-1.8,0) {{ $-\frac{(N-2)}{m}$}};
\node [above] at (-1.8,2) {{$\mathcal{D}$}};
\node [below right] at (0,0) {{ $0$}};
\end{tikzpicture}
\caption{The isocline \eqref{interm8}, the region $\mathcal{D}$ and the orbit going out of $Q_1$ into the region $\mathcal{D}$} 
\end{figure}
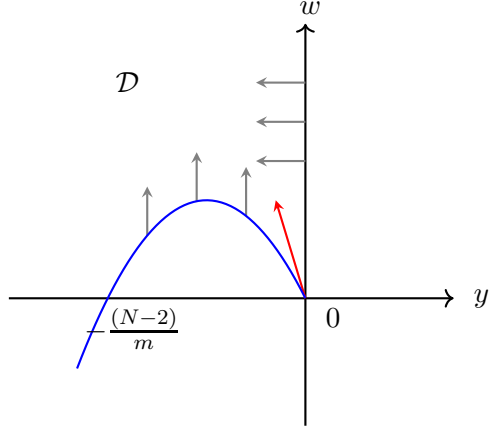

From the definition of $\mathcal{D}$, we observe that along the orbits contained in it, the variable $y$ decreases while the variable $w$ increases, thus there exist
$$
y_0:=\lim\limits_{\eta\to\infty}y(\eta)\in[-\infty,0), \qquad w_0:=\lim\limits_{\eta\to\infty}w(\eta)\in(0,\infty],
$$
where $\eta$ is the independent variable of the system. We immediately see that $(y_0,w_0)$ cannot be a finite point, as it would be a critical point of the system \eqref{syst.lim0} and there is no such a point. Thus, at least one of $(y_0,w_0)$ is infinite; assume for contradiction that one of them is finite. We are left with the following two cases:

$\bullet$ $y_0\in(-\infty,0)$, $w_0=+\infty$. We then deduce from the system \eqref{syst.lim0} that
$$
\frac{dy}{dw}=-\frac{(m-1)[(N-2)y+my^2]}{(m-p)(2-(m-1)y)w}-\frac{(m-1)m^{(1-p)/(m-1)}}{(m-p)(2-(m-1)y)}\to-\frac{(m-1)m^{(1-p)/(m-1)}}{(m-p)(2-(m-1)y_0)}
$$
as $\eta\to\infty$, and such a linear limit behavior obviously contradicts the assumption of a vertical asymptote of the orbit at $y=y_0<0$.

$\bullet$ $y_0=-\infty$, $w_0\in(0,\infty)$. We thus deduce from the system \eqref{syst.lim0} that
$$
\frac{dy}{dw}=-\frac{(m-1)[(N-2)y+my^2]}{(m-p)(2-(m-1)y)w}-\frac{(m-1)m^{(1-p)/(m-1)}}{(m-p)(2-(m-1)y)}\sim\frac{my}{(m-p)w_0},
$$
as $\eta\to\infty$, which by integration gives
$$
y\sim Ke^{mw/(m-p)w_0}\to Ke^{m/(m-p)}\in\real,
$$
contradicting the assumption that $y_0=-\infty$.

We thus infer that $y_0=-\infty$ and $w_0=+\infty$, that is, $y\to-\infty$ and $w\to+\infty$ as $\eta\to\infty$ along the orbit going out of $Q_1$. Going back to variables $(X,Y)$ we obtain that $Y/X\to-\infty$ and $X\to0$ as $\eta\to\infty$. The classification of critical points easily then leads us to the conclusion that this orbit enters the stable node $Q_3$. Since $Q_3$ is a stable node, the conclusion follows from the continuity with respect to the parameter $\alpha$ in the system \eqref{PSsystinf1}.
\end{proof}
We will next study the orbit $l(\alpha)$ for large values of $\alpha$. To this end, we need a slight modification of the initial variables in order to reach a system having the parameter $\alpha$ (or $\beta$, which is proportional) in a denominator. We thus perform the following (easy) change of variables by letting
\begin{equation}\label{PSchange2}
\mathcal{X}(\xi):=\frac{X(\xi)}{\beta}, \qquad \mathcal{Y}(\xi):=\frac{Y(\xi)}{\beta}, \qquad \overline{\eta}(\xi):=\beta\eta(\xi).
\end{equation}
In these variables, the system \eqref{PSsyst1} transforms into a new system
\begin{equation}\label{PSsyst2}
\left\{\begin{array}{ll}\dot{\mathcal{X}}=\mathcal{X}\left[(m-1)\mathcal{Y}-2\mathcal{X}\right],\\[1mm]
\dot{\mathcal{Y}}=-\mathcal{Y}^2-\mathcal{Y}+\frac{2}{m-1}\mathcal{X}-N\mathcal{X}\mathcal{Y}-\frac{m^{(1-p)/(m-1)}}{\beta^{(m-p)/(m-1)}}\mathcal{X}^{(m+p-2)/(m-1)},\end{array}\right.
\end{equation}
Since homotheties are obvious diffeomorphisms, the systems \eqref{PSsyst1} and \eqref{PSsyst2} are topologically equivalent and in fact, their critical points (both finite and infinite) are the same, with similar analysis, except for the point $P_1$ which in terms of the system \eqref{PSsyst2} is seen as $P_1=(0,-1)$. We then let again
\begin{equation}\label{change.inf3}
y=\frac{\mathcal{Y}}{\mathcal{X}}, \qquad z=\frac{1}{\mathcal{X}}
\end{equation}
and obtain the system
\begin{equation}\label{PSsystinf2}
\left\{\begin{array}{ll}\dot{y}=-(N-2)y-my^2-yz+\frac{2}{m-1}z-\frac{m^{(1-p)/(m-1)}}{\beta^{(m-p)/(m-1)}}z^{(m-p)/(m-1)}\\[1mm]
\dot{z}=2z-(m-1)yz,\end{array}\right.
\end{equation}
Letting $\beta\to\infty$ is equivalent, at a formal level, to study the system
\begin{equation}\label{syst.liminf}
\left\{\begin{array}{ll}\dot{y}=-(N-2)y-my^2-yz+\frac{2}{m-1}z\\[1mm]
\dot{z}=2z-(m-1)yz,\end{array}\right.
\end{equation}
which is now quadratic and no further change of variable is needed.
\begin{proposition}\label{prop.large}
There exists $\alpha_2>0$ (and corresponding $\beta_2$) such that for any $\alpha\in(\alpha_2,\infty)$, the orbit $l(\alpha)$ connects to the critical point $P_0$.
\end{proposition}
\begin{proof}
We work on the limit system (at least formally) \eqref{syst.liminf} obtained as $\beta\to\infty$ from \eqref{PSsystinf2}. The linear approximation of the critical point $Q_1=(0,0)$ in the system \eqref{syst.liminf} has the matrix
$$
M=\left(
    \begin{array}{cc}
      -(N-2) & \frac{2}{m-1} \\[1mm]
      0 & 2 \\
    \end{array}
  \right),
$$
and thus the orbit we are interested in, that we will call $l(\infty)$, goes out tangent to the second eigenvector $e_2=(2,N(m-1))$, that is, in the half-plane $\{y>0\}$. The flow of the system \eqref{syst.liminf} over the vertical lines $\{y=0\}$, respectively $\{y=2/(m-1)\}$, in both cases taking as normal direction $\overline{n}=(1,0)$, is given by the signs of the expressions
$$
F_1(z)=\frac{2}{m-1}z>0, \qquad {\rm respectively} \qquad F_2(z)=-\frac{2}{m-1}\left(N-2+\frac{2m}{m-1}\right)<0,
$$
hence the strip $\mathcal{S}=\{(y,z):0<y<2/(m-1),z>0\}$ is positively invariant and the orbit $l(\infty)$ stays forever inside $\mathcal{S}$. Let us now restrict ourselves for the moment to dimension $N\geq2$ and consider the part of the isocline
\begin{equation}\label{interm10}
-(N-2)y-my^2-yz+\frac{2}{m-1}z=0, \qquad {\rm or \ equivalently} \qquad z(y)=\frac{y(N-2+my)}{2/(m-1)-y},
\end{equation}
contained in the strip $\mathcal{S}$. We easily observe that the isocline \eqref{interm10} is increasing as a function $z(y)$, with $z(0)=0$, having a vertical asymptote at $y=2/(m-1)$ and with normal vector $\overline{n}=(-(N-2)-2my-z,2/(m-1)-y)$. The direction of the flow of the system \eqref{syst.liminf} over the curve \eqref{interm10} is given by the sign of the scalar product between this normal vector and the vector field of the system, which gives
$$
G(y,z)=\left(\frac{2}{m-1}-y\right)^2(m-1)z>0.
$$
It follows that the flow goes from the region of the strip $\mathcal{S}$ with $\dot{y}<0$ into the region with $\dot{y}>0$. By comparing the slopes of the eigenvector $e_2$ (which is equal to $N(m-1)/2$) and of the isocline \eqref{interm10} in a small neighborhood of $(y,z)=(0,0)$, which is given by
$$
z'(0)=\frac{(N-2)(m-1)}{2}<\frac{N(m-1)}{2},
$$
we readily deduce that the orbit $l(\infty)$ goes out in the region of the strip $\mathcal{S}$ where $\dot{y}>0$, and thus remains there along all its trajectory. In particular, the coordinate $y$ is increasing along this orbit. The same is true, in an analogous way, in dimension $N=1$. We illustrate the previous arguments, with the regions where $\dot{y}>0$ and $\dot{y}<0$, the isocline \eqref{interm10} and the orbit going out of $Q_1$, in Figure 2.

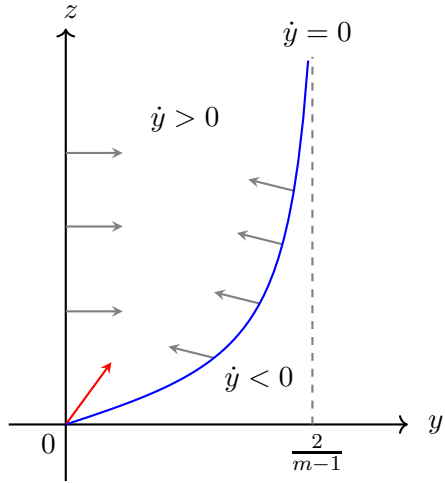
\begin{figure}[h]\label{Fig2}
\centering
\begin{tikzpicture}[scale=0.75]
\draw [->, thick](-1,0) -- (6,0);
\draw [->, thick](0,-1) -- (0,7);
\node [right] at (6,0) {{ $y$}};
\node [above] at (0,7) {{ $z$}};

\draw [red, thick,->,>=stealth] (0,0)--(0.8,1.1);

\draw [gray, thick,->, >=stealth] (0,2)--(1,2);
\draw [gray, thick,->, >=stealth] (0,3.5)--(1,3.5);
\draw [gray, thick,->, >=stealth] (0,4.8)--(1,4.8);

\draw [gray, thick,->, >=stealth] (2.6,1.177)--(1.8,1.377);
\draw [gray, thick,->, >=stealth] (3.4,2.1382)--(2.6,2.3382);
\draw [gray, thick,->, >=stealth] (3.8,3.186)--(3.0,3.386);
\draw [gray, thick,->, >=stealth] (4,4.1317)--(3.2,4.3317);

\draw [blue, domain=0:4.25, thick, samples=50] plot (\x,{tan(\x/3 r)});
\draw [gray, thick,dashed] (4.325,0)--(4.325,6.5);

\node [above] at (4.32,6.5) {{ $\dot{y}=0$}};
\node [above] at (2,5) {{ $\dot{y}>0$}};
\node at (3.3,0.8) {{ $\dot{y}<0$}};

\node [below] at (4.32,0) {{ $\frac{2}{m-1}$}};
\node [below left] at (0,0) {{ $0$}};
\end{tikzpicture}
\caption{Regions where $\dot{y}>0$ and $\dot{y}<0$ in the strip $\mathcal{S}$ and the orbit going out of $Q_1$}
\end{figure}

We also remark that in the strip $\mathcal{S}$, the coordinate $z$ increases along the orbits and thus there exist $y_0:=\lim\limits_{\overline{\eta}\to\infty}y(\overline{\eta})$, which is finite since $y$ is bounded, and $z_0:=\lim\limits_{\overline{\eta}\to\infty}z(\overline{\eta})$. Assuming for contradiction that $z_0\in\real$, it then follows that $(y_0,z_0)$ is a finite critical point of the system \eqref{syst.liminf} lying in the closed strip $\overline{\mathcal{S}}$. But such a point does not exist, and we then conclude that $z_0=+\infty$. Undoing then the change of variable \eqref{change.inf3} we find that
$$
\lim\limits_{\overline{\eta}\to\infty}\mathcal{Y}(\overline{\eta})=\lim\limits_{\overline{\eta}\to\infty}\frac{y(\overline{\eta})}{z(\overline{\eta})}=0, \qquad
\lim\limits_{\overline{\eta}\to\infty}\mathcal{X}(\overline{\eta})=\lim\limits_{\overline{\eta}\to\infty}\frac{1}{z(\overline{\eta})}=0,
$$
which proves that the orbit $l(\infty)$ enters $P_0$. Since $P_0$ has a stable node sector for the orbits arriving from the positive region of the phase plane, as shown in Lemma \ref{lem.P0}, the conclusion follows by standard arguments of continuity with respect to the parameter $\beta$ in the system \eqref{PSsystinf2}.
\end{proof}
We are now in a position to complete the proof of the existence part of Theorem \ref{th.1}.
\begin{proof}[Proof of Theorem \ref{th.1}: existence]
Let us define the following three sets:
\begin{equation*}
\begin{split}
&\mathcal{A}=\{\alpha\in(0,\infty): {\rm the \ orbit} \ l(\alpha) \ {\rm enters} \ Q_3\},\\
&\mathcal{C}=\{\alpha\in(0,\infty): {\rm the \ orbit} \ l(\alpha) \ {\rm enters} \ P_0\},\\
&\mathcal{B}=\{\alpha\in(0,\infty): {\rm the \ orbit} \ l(\alpha) \ {\rm does \ neither \ enter} \ Q_3 \ {\rm nor} \ P_0\}.
\end{split}
\end{equation*}
Proposition \ref{prop.low} together with the fact that $Q_3$ is a stable node imply that the set $\mathcal{A}$ is non-empty and open, while Proposition \ref{prop.large} together with the fact that $P_0$ has a stable node sector in the half-plane $\{X>0\}$ (which is the same as $\mathcal{X}>0$) give that the set $\mathcal{C}$ is non-empty and open. Since the three sets above are obviously disjoint and $\mathcal{A}\cup\mathcal{B}\cup\mathcal{C}=(0,\infty)$, standard topological results imply that the set $\mathcal{B}$ is non-empty and closed. Let then $\alpha\in\mathcal{B}$. According to the Poincar\'e-Bendixon theory \cite[Section 3.7]{Pe}, its $\omega$-limit set must be a critical point, as any periodic orbit should contain a finite critical point in its interior region (see for example \cite[Theorem 5, Section 3.7]{Pe}) and there are no such points. But the only critical point different from $P_0$ and $Q_3$ which is not totally unstable is $P_1$. In conclusion, we have an orbit connecting $Q_1$ and $P_1$ for any $\alpha\in\mathcal{B}$.
\end{proof}

\section{Proof of Theorem \ref{th.1}: uniqueness}\label{sec.uniq}

We assume once more that \eqref{range.exp} and \eqref{range.exp1} if $N=1$ hold true. We next show that the set $\mathcal{B}$ introduced at the end of the previous section is a singleton. To this end, we work throughout this section with the system \eqref{PSsyst2} and the idea of the proof is that the already existing orbit connecting the critical points $Q_1$ and $P_1$ will serve as a barrier for other possible orbits. This will be made rigorous below.
\begin{proof}[Proof of Theorem \ref{th.1}: uniqueness]
We divide this proof into several steps for the reader's convenience.

\medskip

\noindent \textbf{Step 1. Local monotonicity with respect to $\alpha$ on the orbits $l(\alpha)$}. Recall that the orbit $l(\alpha)$ is the one going out of $Q_1$ and containing profiles with local behavior given by \eqref{beh.Q1}. In order to analyze it in general, we start from the system \eqref{PSsystinf2} and we perform the further change of variable $w=z^{(m-p)/(m-1)}$ (also used in Section \ref{sec.infty} to deduce the system \eqref{PSsystinf1}) to find the new system
\begin{equation}\label{PSsystinf3}
\left\{\begin{array}{ll}\dot{y}=-(N-2)y-my^2-yw^{(m-1)/(m-p)}+\frac{2}{m-1}w^{(m-1)/(m-p)}-\frac{m^{(1-p)/(m-1)}}{\beta^{(m-p)/(m-1)}}w\\[1mm]
\dot{w}=\frac{m-p}{m-1}w(2-(m-1)y),\end{array}\right.
\end{equation}
whose vector field is of class $C^1$ and whose linear approximation in a neighborhood of $Q_1=(0,0)$ has the matrix
$$
M_1:=\left(
       \begin{array}{cc}
         -(N-2) & -\frac{m^{(1-p)/(m-1)}}{\beta^{(m-p)/(m-1)}} \\[1mm]
         0 & \frac{2(m-p)}{m-1} \\
       \end{array}
     \right).
$$
We next recall that the orbit $l(\alpha)$ is the one going out tangent to the eigenvector corresponding to the second eigenvalue $\lambda_2=2(m-p)/(m-1)$, that is,
$$
\overline{e}_2:=\left(-\frac{m^{(1-p)/(m-1)}}{\beta^{(m-p)/(m-1)}},\frac{2(m-p)}{m-1}+N-2\right).
$$
This implies that, for fixed $w$, the $y$-coordinate increases with $\beta$ (and thus with $\alpha$) in a neighborhood of $Q_1$, hence, coming back to initial coordinates $(\mathcal{X},\mathcal{Y})$, the coordinate $\mathcal{Y}=y/w^{(m-1)/(m-p)}$ is also increasing with respect to $\alpha$ for fixed values of $\mathcal{X}$.

\medskip

\noindent \textbf{Step 2. Local monotonicity with respect to $\alpha$ on the unique orbit in the stable manifold of $P_1$}. We recall now that $P_1$ corresponds to the critical point $(0,-1)$ in the variables of the system \eqref{PSsyst2}, and this point is a saddle. For the easiness of the calculations, we translate this critical point to the origin by letting $\mathcal{U}:=\mathcal{Y}+1$. Thus, the system \eqref{PSsyst2} writes in variables $(\mathcal{X},\mathcal{U})$ as
\begin{equation}\label{interm11}
\left\{\begin{array}{ll}\dot{\mathcal{X}}=\mathcal{X}\left[(m-1)\mathcal{U}-2\mathcal{X}-(m-1)\right],\\[1mm]
\dot{\mathcal{U}}=-\mathcal{U}^2+\mathcal{U}+\frac{mN-N+2}{m-1}\mathcal{X}-N\mathcal{X}\mathcal{U}-\frac{m^{(1-p)/(m-1)}}{\beta^{(m-p)/(m-1)}}\mathcal{X}^{(m+p-2)/(m-1)},\end{array}\right.
\end{equation}
and our critical point is now the origin of the system \eqref{interm11}. The linear approximation of it in a neighborhood of $(0,0)$ has the matrix
$$
M(0,0)=\left(
         \begin{array}{cc}
           -(m-1) & 0 \\[1mm]
           \frac{mN-N+2}{m-1} & 1 \\
         \end{array}
       \right),
$$
and the unique orbit in the stable manifold is tangent to the eigenvector $e_1=(m,-(mN-N+2)/(m-1))$, hence, in the linear approximation our orbit writes
$$
\mathcal{U}=-\frac{mN-N+2}{m(m-1)}\mathcal{X}+o(\mathcal{X}).
$$
Since we have no explicit dependence on $\alpha$ in this expression, we have to go to the next order in the expansion. To this end, taking into account that $(m+p-2)/(m-1)\in(1,2)$ and the fact that
$$
\mathcal{X}\mathcal{U}=-\frac{mN-N+2}{m(m-1)}\mathcal{X}^2+o(\mathcal{X}^2)=o\Big(\mathcal{X}^{(m+p-2)/(m-1)}\Big),
$$
we can neglect the quadratic terms of the vector field of \eqref{interm11} (which are all $o\Big(\mathcal{X}^{(m+p-2)/(m-1)}\Big)$) to get
\begin{equation}\label{interm12}
\frac{d\mathcal{U}}{d\mathcal{X}}=-\frac{\mathcal{U}+\frac{mN-N+2}{m-1}\mathcal{X}-\frac{m^{(1-p)/(m-1)}}{\beta^{(m-p)/(m-1)}}\mathcal{X}^{(m+p-2)/(m-1)}}{(m-1)\mathcal{X}}
+o\Big(\mathcal{X}^{(p-1)/(m-1)}\Big).
\end{equation}
Integrating \eqref{interm12} gives then
$$
\mathcal{Y}+1=\mathcal{U}=-\frac{mN-N+2}{m(m-1)}\mathcal{X}+\frac{m^{(1-p)/(m-1)}}{(m+p-1)\beta^{(m-p)/(m-1)}}\mathcal{X}^{(m+p-2)/(m-1)}+o\Big(\mathcal{X}^{(m+p-2)/(m-1)}\Big),
$$
and we infer that, if $\mathcal{X}$ is fixed, the coordinate $\mathcal{Y}$ decreases with respect to the exponent $\beta$ (and thus to $\alpha$) in a small neighborhood of $P_1=(0,-1)$.

\medskip

\noindent \textbf{Step 3. Uniqueness of the connection from $Q_1$ to $P_1$}. We have obtained in the previous steps reversed monotone behaviors with respect to $\alpha$ of the orbit $l(\alpha)$, respectively of the stable manifold of $P_1$, but for the moment they are valid only locally near the critical points $Q_1$, respectively $P_1$. Set now
$$
\alpha^*:=\inf\mathcal{B}>0, \qquad \beta^*:=\frac{(m-1)\alpha^*}{2},
$$
which are well defined since we recall that the set $\mathcal{B}$ is non-empty, closed and strictly included in $(0,\infty)$. In particular, $(\alpha^*,\beta^*)$ are the smallest self-similarity exponents for which there exists a connection between $Q_1$ and $P_1$. Consider next, as a last preparatory fact, the line $(m-1)\mathcal{Y}-2\mathcal{X}=0$, which is an isocline of \eqref{PSsyst2}. The direction of the flow of the system \eqref{PSsyst2} over it is given by the sign of the expression
\begin{equation*}
\begin{split}
H(\mathcal{X},\mathcal{Y})&=-(m-1)\mathcal{Y}^2-N(m-1)\mathcal{X}\mathcal{Y}-\frac{(m-1)m^{(1-p)/(m-1)}}{\beta^{(m-p)/(m-1)}}\mathcal{X}^{(m+p-2)/(m-1)}\\
&=-\left[(m-1)+\frac{N(m-1)^2}{2}\right]\mathcal{Y}^2-\frac{(m-1)m^{(1-p)/(m-1)}}{\beta^{(m-p)/(m-1)}}\mathcal{X}^{(m+p-2)/(m-1)}<0.
\end{split}
\end{equation*}
The latter gives that the half-plane $\{(m-1)\mathcal{Y}-2\mathcal{X}<0\}$ is positively invariant for the flow of the system \eqref{PSsyst2} and thus the orbit $l(\alpha)$, going out of $Q_1$ in this half-plane, will stay forever inside it. This implies furthermore that $\mathcal{X}$ is decreasing along this orbit and it allows to write along this orbit that
\begin{equation}\label{monot.global}
\begin{split}
\frac{d\mathcal{Y}}{d\mathcal{X}}&=\frac{1}{(m-1)Y-2X}\\&\times\left[\frac{-(m-1)\mathcal{Y}^2-(m-1)\mathcal{Y}+2\mathcal{X}-N(m-1)\mathcal{X}\mathcal{Y}}{(m-1)\mathcal{X}}-
\frac{m^{(1-p)/(m-1)}}{\beta^{(m-p)/(m-1)}}\mathcal{X}^{(p-1)/(m-1)}\right].
\end{split}
\end{equation}
We notice from \eqref{monot.global} and the fact that $(m-1)\mathcal{Y}-2\mathcal{X}<0$ that the expression of $d\mathcal{Y}/d\mathcal{X}$ decreases with respect to the exponent $\beta$ (and thus with respect to $\alpha$). Step 1 and standard comparison arguments then entail that, for any $\alpha>\alpha^*$, and along the orbit $l(\alpha)$, we have $\mathcal{Y}_{\alpha}>\mathcal{Y}_{\alpha^*}$ for identical values of $\mathcal{X}$ (the order being ensured by the fact that $\mathcal{X}$ decreases along any orbit $l(\alpha)$ with $\alpha>0$, as shown above). This shows that none of the orbits $l(\alpha)$ with $\alpha>\alpha^*$ can cross the orbit $l(\alpha^*)$ at any point. Since the orbit $l(\alpha^*)$ enters $P_1$, the reversed local monotonicity in a neighborhood of the critical point $P_1$ proved in Step 2 ensures that $l(\alpha)$ cannot connect to the stable manifold of $P_1$ for any $\alpha>\alpha^*$, implying the uniqueness of the exponent $\alpha^*$.

\medskip

\noindent \textbf{Step 4. End of the proof}. We have thus proved completely Part 1 in Theorem \ref{th.1}. Moreover, the uniqueness of $\alpha^*$ together with the outcome of Propositions \ref{prop.low} and \ref{prop.large} allow us to write
$$
\mathcal{A}=(0,\alpha^*), \qquad \mathcal{B}=\{\alpha^*\}, \qquad \mathcal{C}=(\alpha^*,\infty).
$$
The definitions of the sets $\mathcal{A}$ and $\mathcal{C}$, together with the local analysis near the critical points $Q_3$ (see Lemma \ref{lem.Q23}) and $P_0$ (see Lemma \ref{lem.P0}), give then both the existence and uniqueness stated in Part 2 of Theorem \ref{th.1} and the non-existence claimed in Part 3 of Theorem \ref{th.1}.
\end{proof}

\section{General global solutions. Proof of Theorem \ref{th.2}}\label{sec.global}

The main idea of this proof is to construct a monotone sequence of approximating solutions and then pass to the limit to obtain the claimed weak solution to Eq. \eqref{eq1} as a limit solution. In this process, the self-similar solutions whose profiles are classified in Theorem \ref{th.1} will be very useful as uniform upper barriers, allowing us to both show that the limit function is finite at every point and that it has the required regularity. Set then, for any $\alpha\geq\alpha^*$ and $\beta=(m-1)\alpha/2$,
$$
U_{\alpha}(x,t):=e^{\alpha t}f(|x|e^{-\beta t})
$$
to be the unique eternal self-similar solution (as established in Theorem \ref{th.1}) having self-similar exponents $\alpha$ and $\beta$ and whose profile satisfies $f(0)=1$.
We begin with a preparatory result.
\begin{lemma}\label{lem.comp}
(a) Let $u_0\in L^{\infty}(\real^N)$. Then, for any $\alpha>\alpha^*$, there exists $\tau_0>0$ (which might depend on $\alpha$) such that
$$
u_0(x)\leq U_{\alpha}(x,\tau_0), \qquad {\rm for \ any} \ x\in\real^N.
$$

\medskip

\noindent (b) Moreover, if $u_0$ is compactly supported, then there exists $\tau^*_0>0$ such that
$$
u_0(x)\leq U_{\alpha^*}(x,\tau^*_0), \qquad {\rm for \ any} \ x\in\real^N.
$$
\end{lemma}
\begin{proof}
(a) Let $u_0\in L^{\infty}(\real^N)$ and pick any $\alpha>\alpha^*$. We infer from the local behaviors \eqref{beh.Q1} as $\xi\to0$ and \eqref{beh.P0} as $\xi\to\infty$ that the profile $f(\xi)$ of the eternal self-similar solution $U_{\alpha}$ is decreasing with respect to $\xi$ in a right neighborhood of $\xi=0$ and tends to infinity as $\xi\to\infty$, thus it has a positive minimum at some point $\xi_{\min}>0$. If we fix $t>0$, we then observe that $U_{\alpha}(t)$ achieves a positive minimum at points $x\in\real^N$ such that $|x|=\xi_{\min}e^{\beta t}$, hence we find that
$$
U_{\alpha}(x,\tau_0)\geq e^{\alpha\tau_0}f(\xi_{\min})\geq\|u_0\|_{\infty}\geq u_0(x),\qquad x\in\real^N,
$$
if we pick $\tau_0>0$ sufficiently large such that
$$
\tau_0\geq\frac{1}{\alpha}\ln\left(\frac{\|u_0\|_{\infty}}{f(\xi_{\min})}\right).
$$

\medskip

\noindent (b) Let $u_0$ be now a compactly supported function, more precisely, ${\rm supp}\,u_0\subset B(0,R)$ for some $R>0$. Take $\alpha=\alpha^*$ and let $\xi_0\in(0,\infty)$ be the edge of the support of the profile $f(\xi)$ of the eternal compactly supported self-similar solution $U_{\alpha^*}$. Set
$$
Q:=\inf\{f(\xi):\xi\in(0,\xi_0/2)\}>0
$$
and choose
\begin{equation}\label{interm13}
\tau_0\geq\max\left\{\frac{1}{\alpha^*}\ln\left(\frac{\|u_0\|_{\infty}}{Q}\right),\frac{1}{\beta^*}\ln\left(\frac{2R}{\xi_0}\right),0\right\}.
\end{equation}
On the one hand, we infer from \eqref{interm13} that
$$
|x|e^{-\beta^*\tau_0}\leq\frac{\xi_0}{2}, \qquad {\rm for \ any} \ x\in B(0,R),
$$
which gives that, on the other hand,
$$
U_{\alpha^*}(x,\tau_0)=e^{\alpha^*\tau_0}f(|x|e^{-\beta^*\tau_0})\geq e^{\alpha^*\tau_0}Q\geq\|u_0\|_{\infty}\geq u_0(x),
$$
again for any $x\in B(0,R)$. The conclusion follows from the fact that $U_{\alpha^*}(x,\tau_0)\geq0=u_0(x)$ for any $x\in\real^N\setminus B(0,R)$.
\end{proof}
We are now ready to complete the proof of Theorem \ref{th.2}.
\begin{proof}[Proof of Theorem \ref{th.2}]
Let $u_0\in L^{\infty}(\real^N)$ and consider the following approximating family of Cauchy problems
\begin{subequations}\label{Cauchy.eps}
\begin{equation}
\partial_tu_{\epsilon}=\Delta u_{\epsilon}^m+(|x|+\epsilon)^{\sigma}u_{\epsilon}^p, \qquad (x,t)\in\real^N\times(0,\infty), \label{ueps} 
\end{equation}
\begin{equation}
u(x,0)=u_0(x), \qquad x\in\real^N. \label{ueps.init}
\end{equation}
\end{subequations}
for any $\epsilon\in(0,1]$ and for the same exponents $m$, $p$, $\sigma=\sigma_*$ as in \eqref{range.exp} and \eqref{range.exp1}. On the one hand, observe that, if $u_1$ is a solution to \eqref{ueps} with $\epsilon=1$, the following rescaling
\begin{equation}\label{resc.eps}
u_{\epsilon}(x,t)=\epsilon^{2/(m-1)}u_1(\epsilon^{-1}x,t), \qquad (x,t)\in\real^N\times(0,\infty),
\end{equation}
produces a solution to \eqref{ueps} for any $\epsilon\in(0,1)$. On the other hand, since $u_0\in L^{\infty}(\real^N)$, we readily notice that
\begin{equation*}
[[u_0]]_1:=\sup\limits_{\varrho\geq1}\varrho^{-2/(m-1)}\frac{1}{|B(0,\varrho)|}\int_{B(0,\varrho)}|u_0(y)|\,dy\leq\|u_0\|_{\infty}<\infty,
\end{equation*}
hence the existence theory given in \cite[Theorem 3.1]{AdB91} for Eq. \eqref{ueps} with $\epsilon=1$ applies in our case and can be straightforwardly extended to any $\epsilon\in(0,1)$ by the rescaling \eqref{resc.eps}. Moreover, since the coefficients in \eqref{ueps} are bounded, standard theory of parabolic equations (see for example \cite[Section 8, Chapter 5]{LSU}, or also \cite[Section 5]{AdB91}) implies that the uniqueness and the parabolic comparison principle hold true for \eqref{ueps} with $\epsilon\in(0,1)$. We thus infer that there exists a unique weak solution $u_{\epsilon}$ to the Cauchy problem \eqref{Cauchy.eps}, and in particular it satisfies the weak formulation which implies that for any $\varphi\in C_0^{2,1}(\real^N\times(0,T))$ and for any $t_1$, $t_2\in(0,T)$ with $t_1<t_2$ we have
\begin{equation}\label{weak.eps}
\begin{split}
\int_{\real^N}u_{\epsilon}(t_2)\varphi(t_2)\,dx&-\int_{\real^N}u_{\epsilon}(t_1)\varphi(t_1)\,dx-\int_{t_1}^{t_2}\int_{\real^N}u_{\epsilon}(t)\varphi_t(t)\,dx\,dt\\
&-\int_{t_1}^{t_2}\int_{\real^N}u_{\epsilon}^m(t)\Delta\varphi(t)\,dx\,dt=\int_{t_1}^{t_2}\int_{\real^N}(|x|+\epsilon)^{\sigma}u_{\epsilon}^p(t)\varphi(t)\,dx\,dt.
\end{split}
\end{equation}
and the initial condition is taken in $L^1_{\rm loc}$ sense. Moreover, if $0<\epsilon_1<\epsilon_2<1$, we observe that the solution $u_{\epsilon_1}$ is a supersolution to Eq. \eqref{ueps} with $\epsilon=\epsilon_2$, and the comparison principle then entails that $u_{\epsilon_1}(x,t)\geq u_{\epsilon_2}(x,t)$ for any $(x,t)\in\real^N\times(0,\infty)$. Finally, the last element in the proof follows from Lemma \ref{lem.comp}. Indeed, any eternal self-similar solution $U_{\alpha}(\cdot,\cdot+\tau_0)$ with $\alpha\geq\alpha^*$ and $\tau_0\geq0$ given by Lemma \ref{lem.comp} is a supersolution to any of the Cauchy problems \eqref{Cauchy.eps}, and the comparison principle gives that
\begin{equation}\label{interm14}
u_{\epsilon}(x,t)\leq U_{\alpha}(x,t+\tau_0), \qquad (x,t)\in\real^N\times(0,\infty), \qquad \epsilon\in(0,1).
\end{equation}
We deduce that the pointwise limit
\begin{equation}\label{lim.sol}
u(x,t):=\lim\limits_{\epsilon\to0}u_{\epsilon}(x,t), \qquad (x,t)\in\real^N\times(0,\infty)
\end{equation}
is well defined and finite at any point. Moreover, we can pass to the limit as $\epsilon\to0$ in the weak formulation \eqref{weak.eps} and Lebesgue's dominated convergence theorem implies that the limit function $u$ defined in \eqref{lim.sol} satisfies the weak formulation \eqref{weak} of Eq. \eqref{eq1}. In a similar way, the initial condition $u_0$ is taken in $L^1_{\rm loc}$ sense (as explained in Definition \ref{def.sol}) by $u$. It thus follows that $u$ is a weak solution, according to Definition \ref{def.sol}, to the Cauchy problem associated to Eq. \eqref{eq1} with initial condition $u_0$. The fact that
$$
u(t)\in L^{\infty}(\real^N;|x|^{2/(m-1)}), \qquad {\rm for \ any} \ t>0,
$$
if $u_0\in L^{\infty}(\real^N)$, or the more precise property of compact support of $u(t)$ if $u_0$ is compactly supported, follow from the inequality \eqref{interm14} together with the form of the self-similar solutions $U_{\alpha}$ with $\alpha>\alpha^*$ in the former case, respectively $U_{\alpha^*}$ in the latter case, completing the proof.
\end{proof}

\medskip

\noindent \textbf{Remark.} The global existence result in Theorem \ref{th.2} can be extended (with slightly more technical work for the proof of the preparatory Lemma \ref{lem.comp} in this case) with a proof along the same lines as above for any initial condition $u_0\in L^{\infty}(\real^N;|x|^{2/(m-1)})$ or even in more precise logarithmic spaces taking into account also the logarithmic correction in the local behavior \eqref{beh.P0}. We omit the details here.

\bigskip

\noindent \textbf{Acknowledgements} R. I. and A. S. are partially supported by the Spanish project PID2020-115273GB-I00.

\bibliographystyle{plain}

\begin{thebibliography}{1}

\bibitem{Amann}
H. Amann, \emph{Ordinary differential equations. An introduction to nonlinear analysis}, De Gruyter Studies in Mathematics, vol. 13, Berlin, 1990.

\bibitem{AdB91}
D. Andreucci and E. DiBenedetto, \emph{On the Cauchy problem and initial traces for a class of evolution equations with strongly nonlinear sources}, Ann. Scuola Norm. Sup. Pisa, \textbf{18} (1991).

\bibitem{AT05}
D. Andreucci and A. F. Tedeev, \emph{Universal bounds at the
blow-up time for nonlinear parabolic equations}, Adv. Differential
Equations, \textbf{10} (2005), no. 1, 89--120.


%

\bibitem{BG84}
P. Baras and J. Goldstein, \emph{The heat equation with a singular potential}, Trans. Amer. Math. Soc., \textbf{284} (1984), no. 1, 121--139.


\bibitem{BS19}
B. Ben Slimene, \emph{Asymptotically self-similar global solutions for Hardy-H\'enon parabolic systems}, Differ. Equ. Appl., \textbf{11} (2019), no. 4, 439--462.

\bibitem{BSTW17}
B. Ben Slimene, S. Tayachi and F. B. Weissler, \emph{Well-posedness, global existence and large time behavior for Hardy-H\'enon parabolic equations}, Nonlinear Anal., \textbf{152} (2017), 116--148.

\bibitem{CM99}
X. Cabr\'e and Y. Martel, \emph{Existence versus explosion instantan\'ee por des \'equations de la chaleur lin\'eaires avec potentiel singulier}, C. R. Acad. Sci. Paris, \textbf{329} (1999), no. 11, 973--978.

\bibitem{Carr}
J. Carr, \emph{Applications of Centre Manifold Theory}, Springer Verlag, New York, 1981.

\bibitem{CIT21a}
N. Chikami, M. Ikeda and K. Taniguchi, \emph{Well-posedness and global dynamics for the critical Hardy-Sobolev parabolic equation}, Nonlinearity, \textbf{34} (2021), no. 11, 8094--8142.

\bibitem{CIT21b}
N. Chikami, M. Ikeda and K. Taniguchi, \emph{Optimal well-posedness and forward self-similar solution for the Hardy-H\'enon parabolic equation in critical weighted Lebesgue spaces}, Nonlinear Anal., \textbf{222} (2022), Article no. 112931, 28 pp.

%

\bibitem{DS06}
P.~Daskalopoulos and N.~Sesum, \emph{Eternal solutions to the Ricci flow on $\real^2$}, Int. Math. Res. Not.,  2006, Article no. 83610, 20 pp.

%

\bibitem{FT00}
S. Filippas and A. Tertikas, \emph{On similarity solutions of a heat equation with a nonhomogeneous nonlinearity}, J. Differential Equations, \textbf{165} (2000), no. 2, 468--492.

\bibitem{GPV00}
V. A. Galaktionov, L. A. Peletier and J. L. V\'azquez, \emph{Asymptotics of
the fast-diffusion equation with critical exponent}, SIAM J. Math.
Anal. \textbf{31} (2000), no. 5, 1157--1174.

%

\bibitem{GK03}
J. A. Goldstein and I. Kombe, \emph{Nonlinear degenerate prabolic equations with singular lower-order term}, Adv. Differential Equations, \textbf{8} (2003), no. 10, 1153--1192.

\bibitem{GGK05}
G. R. Goldstein, J. A. Goldstein, and I. Kombe, \emph{Nonlinear parabolic equations with singular coefficient and critical exponent}, Appl. Anal., \textbf{84} (2005), no. 6, 571--583.

\bibitem{GH}
J. Guckenheimer and Ph. Holmes, \emph{Nonlinear oscillation, dynamical systems and bifurcations of vector fields}, Applied Mathematical Sciences, vol. 42, Springer-Verlag, New York, 1990.

%
%
%

\bibitem{HT21}
K. Hisa and J. Takahashi, \emph{Optimal singularities of initial data for solvability of the Hardy parabolic equation}, J. Differential Equations, \textbf{296} (2021), 822--848.

\bibitem{IL13a}
R. G. Iagar and Ph.~Lauren\ced{c}ot, \emph{Existence and uniqueness
of very singular solutions for a fast diffusion equation with
gradient absorption},  J. London Math. Soc., \textbf{87} (2013),
509--529.

\bibitem{IL13b}
R. G. Iagar and Ph. Lauren\ced{c}ot, \emph{Eternal solutions to a singular diffusion equation with critical gradient absorption}, Nonlinearity, \textbf{26} (2013), no. 12, 3169--3195.


\bibitem{ILS22b}
R. G. Iagar, M. Latorre and A. S\'anchez, \emph{Blow-up patterns for a reaction-diffusion equation with weighted reaction in general dimension}, Submitted (2022), Preprint ArXiv no. 2205.09407.

\bibitem{IMS21}
R. G. Iagar, A. I. Mu\~{n}oz and A. S\'anchez, \emph{Self-similar solutions preventing finite time blow-up for reaction-diffusion equations with singular potential}, Submitted (2021), Preprint ArXiv no. 2111.04806.


\bibitem{IS19}
R. G. Iagar and A. S\'anchez, \emph{Blow up profiles for a quasilinear reaction-diffusion equation with weighted reaction with linear growth}, J. Dynam. Differential Equations, \textbf{31} (2019), no. 4, 2061--2094.

\bibitem{IS20}
R. G. Iagar and A. S\'anchez, \emph{Instantaneous and finite time blow-up of solutions to a reaction-diffusion equation with Hardy-type singular potential}, J. Math. Anal. Appl., \textbf{491} (2020), no. 1, Article no. 124244, 11 pages.

\bibitem{IS21}
R. G. Iagar and A. S\'anchez, \emph{Blow up profiles for a quasilinear reaction-diffusion equation with weighted reaction}, J. Differential Equations, \textbf{272} (2021), no. 1, 560--605.


\bibitem{IS22a}
R. G. Iagar and A. S\'anchez, \emph{Eternal solutions for a reaction-diffusion equation with weighted reaction}, Discrete Contin. Dyn. Syst, \textbf{42} (2022), no. 3, 1465--1491.

\bibitem{IS22b}
R. G. Iagar and A. S\'anchez, \emph{Anomalous self-similar solutions of exponential type for the subcritical fast diffusion equation with weighted reaction}, Nonlinearity, \textbf{35} (2022), no. 7, 3385--3416.

\bibitem{IS22c}
R. G. Iagar and A. S\'anchez, \emph{A special self-similar solution and existence of global solutions for a reaction-diffusion equation with Hardy potential}, J. Math. Anal. Appl., \textbf{517} (2023), no. 1, Article no. 126588, 22 pp.




\bibitem{Ko04}
I. Kombe, \emph{Doubly nonlinear parabolic equations with singular lower order term}, Nonlinear Anal., \textbf{56} (2004), no. 2, 185--199.

\bibitem{LSU}
O. A. Ladyzhenskaya, V. A. Solonnikov and N. N. Uraltseva, \emph{Linear and quasi-linear equations of parabolic type}, Translations of Mathematical Monographs, Volume 23, American Mathematical Society, 1988.

%

\bibitem{MS21}
A. Mukai and Y. Seki, \emph{Refined construction of Type II blow-up solutions for semilinear heat equations with Joseph-Lundgren supercritical nonlinearity}, Discrete Cont. Dynamical Systems, \textbf{41} (2021), no. 10, 4847--4885.


\bibitem{Pe}
L. Perko, \emph{Differential equations and dynamical systems. Third
edition}, Texts in Applied Mathematics, \textbf{7}, Springer Verlag,
New York, 2001.

%

\bibitem{Qi98}
Y.-W. Qi, \emph{The critical exponents of parabolic equations and blow-up in $\real^n$}, Proc. Royal Soc. Edinburgh A, \textbf{128} (1998), 123--136.


\bibitem{S4}
A. A. Samarskii, V. A. Galaktionov, S. P. Kurdyumov and A. P.
Mikhailov, \emph{Blow-up in quasilinear parabolic problems}, de
Gruyter Expositions in Mathematics, \textbf{19}, W. de Gruyter,
Berlin, 1995.

\bibitem{S73}
J. Sotomayor, \emph{Generic bifurcations of dynamical systems}, in Proceedings of a Symposium Held at University of Bahia, Salvador, Brasil, Academic Press, New York, 1973, 561--582.

\bibitem{Su02}
R. Suzuki, \emph{Existence and nonexistence of global solutions of
quasilinear parabolic equations}, J. Math. Soc. Japan, \textbf{54}
(2002), no. 4, 747--792.

\bibitem{T20}
S. Tayachi, \emph{Uniqueness and non-uniqueness of solutions for critical Hardy-H\'enon parabolic equations}, J. Math. Anal. Appl. \textbf{488} (2020), no. 1, Article no. 123976, 51 pages.

\end{thebibliography}

\end{document}